\documentclass[11pt]{article}
\usepackage{times}
\usepackage{amssymb}
\usepackage{amsmath}
\usepackage{amsthm}
\usepackage{float,xspace}
\usepackage{caption}
\usepackage{url}
\usepackage{xcolor}
\usepackage{comment}
\usepackage[T1]{fontenc}
\usepackage{caption}
\usepackage{enumerate}
\restylefloat{table}

\usepackage[margin=1.0in]{geometry}

\ProvideTextCommand{\Dbar}{T1}{\DJ}

\ProvideTextCommandDefault{\Dbar}{%
\leavevmode\lower.5ex\rlap{\hskip-.07em\accent"16}D\hspace{-4pt}}

\newtheorem{theorem}{Theorem}[section]
\newtheorem{lemma}[theorem]{Lemma}
\newtheorem{proposition}[theorem]{Proposition}
\newtheorem{corollary}[theorem]{Corollary}
\theoremstyle{definition}
\newtheorem{definition}[theorem]{Definition}
\theoremstyle{example}
\newtheorem{example}[theorem]{Example}
\newtheorem{problem}[theorem]{Problem}
\theoremstyle{remark}
\newtheorem{remark}[theorem]{Remark}

\newcommand{\F}{\mathbb{F}}

\newcommand{\Z}{\mathbb{Z}}

\newcommand{\CGW}{\mathrm{CGW}}

\newcommand{\BH}{\mathrm{BH}}

\newcommand{\PPAF}{\mathrm{PAF}}

\newcommand{\WPPGP}{\mathrm{WPGP}}
\newcommand{\Mod}[1]{\ \mathrm{mod}\ #1}

\begin{document}

\title{A survey of complex generalized weighing matrices and a construction of quantum error-correcting codes}

\author{
\textsc{Ronan Egan}
				\thanks{\textit{E-mail: ronan.egan@dcu.ie}}\\
\textit{\footnotesize{School of Mathematical Sciences}}\\
\textit{\footnotesize{Dublin City University}}\\
}

\maketitle

\begin{abstract}
Some combinatorial designs, such as Hadamard matrices, have been extensively researched and are familiar to readers across the spectrum of Science and Engineering. They arise in diverse fields such as cryptography, communication theory, and quantum computing. Objects like this also lend themselves to compelling mathematics problems, such as the Hadamard conjecture. However, complex generalized weighing matrices, which generalize Hadamard matrices, have not received anything like the same level of scrutiny. Motivated by an application to the construction of quantum error-correcting codes, which we outline in the latter sections of this paper, we survey the existing literature on complex generalized weighing matrices. We discuss and extend upon the known existence conditions and constructions, and compile known existence results for small parameters. Some interesting quantum codes are constructed to demonstrate their value.
\end{abstract}

\footnote{2010 Mathematics Subject Classification: 05B20}

\footnote{Keywords: Complex generalized weighing matrix, Butson Hadamard matrix, Self-orthogonal code, Quantum code}

\section{Introduction}\label{intro}

Combinatorial designs are finite objects that obey some kind of combinatorial condition and they take many forms. Many of them are comprehensively surveyed in the Handbook of Combinatorial Designs \cite{HandbookCD}, to which we refer the reader for more information on almost every design mentioned in this paper. Some designs, such as Hadamard matrices due either to their applications or their appearance in other fields, have been extensively researched due to applications in diverse fields such as cryptography, communication theory, and quantum computing \cite{Horadam}. Objects like this also lend themselves to compelling mathematics problems - the Hadamard conjecture proposing the existence of a Hadamard matrix of order $4n$ for all $n \in \mathbb{N}$ has captured the imagination of numerous researchers since it was posed by Paley almost a century ago \cite{Paley1933}. Other designs are well known only to researchers in closely related fields. Complex generalized weighing matrices, which include Hadamard matrices as a special case, are the subject of this survey. We focus entirely on the general case, and do not survey the extensive literature on the special cases here. This addresses what we feel is a notable gap in the literature, as complex generalised weighing matrices do not appear to have been surveyed elsewhere. They are referenced as an example of a pairwise combinatorial design in de Launey and Flannery's monograph on Algebraic Design Theory \cite{deLF}, but are not analysed in any detail except in the context of results that apply to large families of pairwise combinatorial designs.

To begin, we give the necessary definitions and outline our notation. Thoughout, $k$ is a positive integer and $\zeta_{k} = e^{\frac{2\pi\sqrt{-1}}{k}}$ is a primitive $k^{\rm th}$ root of unity. Let $\langle \zeta_{k} \rangle = \{\zeta_{k}^{j} \; : \; 0 \leq j \leq k-1\}$ be the set of all $k^{\rm th}$ roots of unity, and let $\mathcal{U}_{k} = \{0\} \cup \langle \zeta_{k} \rangle$. We denote the set of all $n \times n$ matrices over the complex numbers by $\mathcal{M}_{n}(\mathbb{C})$, and the subset of matrices with entries in $\mathcal{U}_{k}$ by $\mathcal{M}_{n}(k)$. A monomial matrix is one with exactly one non-zero entry in each row and column. The subset of monomial matrices in $\mathcal{M}_{n}(k)$ is denoted by $\mathrm{Mon}_{n}(k)$. More generally, the set of all $n \times n$ matrices with entries in an alphabet $\mathcal{A}$ containing a zero is denoted by $\mathcal{M}_{n}(\mathcal{A})$, and the set of monomial matrices therein by $\mathrm{Mon}_{n}(\mathcal{A})$. Given a matrix $M \in \mathcal{M}_{n}(\mathcal{A})$, the matrix $S$ obtained from $M$ by replacing each non-zero entry with $1$ is called the \emph{support matrix} of $M$. If $S$ supports $M$, we also say that $S$ \emph{lifts} to $M$.

Given an alphabet $\mathcal{A}$ with the property that $a^{-1} \in \mathcal{A}$ for any non-zero $a \in \mathcal{A}$, we let $\ast$ be the transposition acting on $\mathcal{A}$ such that $a^{\ast} = a^{-1}$ if $a \neq 0$, and $0^{\ast} = 0$. Typically, an alphabet will either be a field, or the set $\mathcal{U}_{k}$.

By a common abuse of notation, when $M = [m_{ij}] \in \mathcal{M}_{n}(\mathcal{A})$, we write $M^{\ast} = [m_{ji}^{\ast}]$ for the Hermitian transpose of $M$. For a complex number $z$, the complex conjugate is denoted by $\overline{z}$. We denote the ring of integers modulo $k$ by $\mathbb{Z}_{k}$. When $k = p$ is prime, the transposition $\ast$ acts on $\mathbb{Z}_{p}$ so that $a^{\ast}$ is the multiplicative inverse of $a$, for all $a \neq 0$. When $q = p^r$ for some prime $p$ and positive integer $r$, we denote by $\mathbb{F}_{q}$, the finite field of order $q$.

Rows and columns of $n \times n$ matrices or sequences of length $n$ are typically indexed by the integers $0,1,\ldots,n-1$. A circulant matrix is an $n \times n$ matrix $C$ which is generated by its first row $[c_{0},c_{1},\ldots,c_{n-1}]$, and is denoted by $C = \mathrm{circ}([c_{0},c_{1},\ldots,c_{n-1}])$. Each row is obtained by shifting the previous row forward cyclically. That is, $C = [c_{j-i}]_{i,j \in \mathbb{Z}_{n}}$.

\begin{definition}
An $n \times n$ matrix $W$ with entries in $\mathcal{U}_{k}$ is a \emph{complex generalized weighing matrix} of \emph{weight} $w$ if $WW^{*} = wI_{n}$, where $W^{*}$ is the complex conjugate transpose of $W$, and $I_{n}$ denotes the $n \times n$ identity matrix. The set of all such matrices is denoted by $\mathrm{CGW}(n,w;k)$.
\end{definition}

We will abbreviate to CGW when parameters are unspecified. An $n \times n$ matrix $W$ is a $\mathrm{CGW}(n,w;k)$ if each row of $W$ has exactly $w$ non-zero entries in $\langle \zeta_{k} \rangle$, and the Hermitian inner product of any two distinct rows in zero.  If $w=n$, then $W$ is a \emph{Butson Hadamard matrix}, the set of which is denoted by $\mathrm{BH}(n,k)$. If $k = 2$, then $W$ is a \emph{weighing matrix}, the set of which is denoted by $W(n,w)$. If both $k=2$ and $w=n$, then $W$ is a \emph{Hadamard matrix}, the set of which is denoted by $\mathrm{H}(n)$. Hadamard matrices in particular have been studied extensively for over a century, for detailed expositions we refer the reader to any of \cite{Agaian,BJL,HandbookCD,CraigenSG,HallCombinatorialTheory,Horadam,SeberryOrthogonalDesigns,SeberryYamada}.
Both weighing matrices and Butson Hadamard matrices have been studied frequently, albeit far less than the Hadamard matrices that comprise their intersection. Weighing matrices feature prominently in works of Craigen, Seberry, and their coauthors; see any of \cite{CraigenWeaving93,CraigenWeaving95,CraigendeLauney,CraigenPhD,CraigenKha,SeberryOrthogonalDesigns,SeberryWhiteman,SeberryYamada} for examples. For a general background into Butson Hadamard matrices, which are often called generalized Hadamard matrices by different authors, see any of \cite{Butson,deLF,PUGPpaper,BHPaper2,MorphismsPaper,LSO,FerencThesis}, and for a comprehensive survey and an up to date catalog, see \cite{CHGuide} and \cite{Karol} respectively.

Despite being the superset containing all of these objects, CGWs have, in their own right, received very little scrutiny outside of these special cases. The first significant work we note is due to Berman \cite{Berman77,Berman78}. Berman's constructions, which we will discuss in this paper, reveal several connections to finite geometry and finite fields, and demonstrate that these objects merit study outside of the Butson Hadamard or real weighing matrix cases. Around the same time, Seberry \cite{Seberry79} and Seberry and Whiteman \cite{SeberryWhiteman} considered the case where $k=4$ due to a relationship to orthogonal designs. We document these in Section \ref{section:Construct}. Only sporadic work on the topic has since appeared, perhaps most significantly due to Craigen and de Launey \cite{CraigendeLauney}, who studied CGWs that are invariant under a regular group action. To our knowledge, there has been no recent comprehensive survey collating up to date results on CGWs, that do not qualify as being either real weighing matrices, or Butson Hadamard matrices.

In Section \ref{section:Props} we discuss and extend upon the known existence conditions. In Section \ref{section:Construct} we describe the known constructions, beginning with direct constructions and then recursive constructions. In some cases, known constructions of objects like weighing matrices are generalized. In Section \ref{section:Quant} we introduce a construction of Hermitian self-orthogonal $q$-ary codes from appropriate CGWs and describe the subsequent approach to building quantum codes. This application motivates our survey. In Section \ref{section:results} we report on early computational results from this construction. Finally, an appendix follows the paper containing tables collating the information of Sections \ref{section:Props} and \ref{section:Construct}, giving existence or nonexistence of $\CGW(n,w,k)$, if known, for all $1 \leq n \leq 15$, $1 \leq w \leq n$ and $2 \leq k \leq 6$.

\section{Existence conditions}\label{section:Props}

Existence conditions for $\CGW(n,w;k)$ tend to be number theoretical, but we will combine these with techniques from design theory. A \emph{generalized weighing matrix} $W$ is an $n\times n$ matrix with $w$ non-zero entries in each row and column coming from a finite group $G$ such that $WW^{\ast} = wI_{n}$ over $\Z[G]/\Z G$. Because generalized weighing matrices over groups of prime order $k$ coincide with CGWs over $\mathcal{U}_{k}$ (see, e.g., \cite[Lemma 2.2]{BHPaper2}), non-existence results for generalized weighing matrices over groups of prime order will apply.

\subsection{Equivalence}

The group action of $\mathrm{Mon}_{n}^{2}(k)$ on a matrix in $M \in \mathcal{M}_{n}(k)$ is defined by $M(P,Q) = PMQ^{*}$. This action stabilizes the set $\mathrm{CGW}(n,w;k)$.
The orbit of $W$ under this action is the \emph{equivalence class} of $W$. That is, any matrix $W'$ obtainable from $W$ by permuting rows (respectively columns) or multiplying rows (respectively columns) by an element of $\langle \zeta_{k} \rangle$ is also an element of $\mathrm{CGW}(n,w;k)$, and is said to be \emph{equivalent} to $W$. More succinctly, two matrices $W$ and $W'$ are equivalent if
\[
W' = PWQ^{*}
\]
for matrices $P,Q \in \mathrm{Mon}_{k}(n)$, and we write $W \equiv W'$. It is typical to study the set $\mathrm{CGW}(n,w;k)$ through the lens of equivalence. It is a particularly useful tool when proving the non-existence of an element of $\mathrm{CGW}(n,w;k)$. Complete classifications up to equivalence, even for reasonably small $n$, are computationally difficult. Equivalence classes of Hadamard matrices of order up to $32$ have been classified, but the number of classes appears to grow extremely rapidly with $n$, and the problem becomes computationally infeasible very quickly. Classifying Butson matrices or weighing matrices is even more difficult. Harada et al. \cite{HLMT} used coding theory techniques to classify $\BH(18,3)$ up to equivalence. The equivalence classes of $\BH(n,4)$ for $n \in \{10,12,14\}$ were determined by Lampio, Sz\"oll{\H{o}}si and \"Osterg{\aa}rd in \cite{LSO}, and they later classified $\BH(21,3)$, $\BH(16,4)$ and $\BH(14,6)$ in \cite{LOS}, using computational methods.  Other classifications restrict to matrices with extra properties. Group developed and cocyclic matrices are examples of special cases that have a lot of extra structure, reducing the search space significantly. The cocyclic equivalence classes in $\BH(n,p)$ were classified for all $np \leq 100$ where $p$ is an odd prime in \cite{BHPaper2}, and the cocyclic real Hadamard matrices have been classified at all orders up to $36$ in \cite{POCRod}, and at orders $44$ and $52$ in \cite{POCSanHei}.

\subsection{Number theoretical conditions}

The inner product of any pair of distinct rows of columns of a CGW must equal zero, hence the main number theoretical restrictions follow from a Theorem of Lam and Leung on vanishing sums of roots of unity \cite{LamLeung}.

\begin{theorem}\label{LamLeungVan}
If $\sum_{j=0}^{k-1} c_{j}\zeta_{k}^{j} = 0$ for non-negative integers $c_{0},\ldots, c_{k-1}$, and $p_1,\dots,p_r$ are the primes
dividing $k$, then $\sum_{j=0}^{k-1} c_{j} = \sum_{\ell=1}^{r}d_{\ell}p_{\ell}$ where
$d_1,\dots,d_{\ell}$ are non-negative integers.
\end{theorem}

A special case of this is the well known fact that if $\sum_{i=0}^{p-1} c_{i}\zeta_{p}^{i} = 0$ for a prime $p$, then $\sum_{i=0}^{p-1} c_{i} = dp$ for some non-negative integer $d$, and consequently $c_{0} = c_{1} = \cdots = c_{p-1} = d$. Hence, a $\CGW(n,w;p^r)$ exists only if the non-zero entries in any two distinct rows or columns coincide in a multiple of $p$ positions.

The question of existence is more complicated when $k$ is composite, as Theorem \ref{LamLeungVan} is a less significant barrier. When $k=6$, it is no restriction at all, but some further conditions on their existence are described below. Formulating precise existence conditions in this case is far from straightforward - there is a known element of $\BH(7,6)$ so the order can be coprime to $6$, but the set $\BH(5,6)$ is empty. Perhaps the most general nonexistence results follow from a condition on the determinant of the gram matrix.

\begin{lemma}\label{lem:norms}
Suppose that there exists $W \in \CGW(n,w;k)$. Then $|\mathrm{det}(W)|^2 = w^n$.
\end{lemma}

\begin{proof}
If $W \in \CGW(n,w;k)$ then $WW^{\ast} = wI_{n}$. It follows that
\[
|\mathrm{det}(W)|^2 = \mathrm{det}(W)\mathrm{det}(W^{\ast}) = \mathrm{det}(WW^{\ast}) = w^n.
\]
\end{proof}

Lemma \ref{lem:norms} implies the well known conditions that a $\CGW(n,w;2)$ exists when $n$ is odd only if $w$ is a square. The Sum of Two Squares Theorem states that $w$ is expressible as the sum of two squares if and only if the square free part of $w$ is not divisible by any prime $p \equiv 3 \mod 4$. It follows that a $\CGW(n,w;4)$ exists when $n$ is odd only if $w$ is the sum of two integer squares. We will see Lemma \ref{lem:norms} applied again later in this section.

Many of the better known, and easiest to apply non-existence results apply when $k$ is prime.
In the case of real weighing matrices, many of the strongest non-existence conditions were described in the 1970's by Geramita, Geramita and Seberry Wallis \cite{GGSW}. When $k \neq 2$, some of these results have been generalized, and extended. Some of the best known conditions when $k$ is an odd prime are due to de Launey \cite{deLauneyExistence}, who studied generalized weighing matrices. For consistency, we present the relevant results in the language of CGWs.

\begin{theorem}[cf. Theorem 1.2, \cite{deLauneyExistence}]\label{deLNE2}
If there exists a $\CGW(n,w;k)$ with $n \neq w$ and $k$ a prime, then the following must hold:
\begin{enumerate}[(i)]
\item $w(w-1) \equiv 0 \Mod k$.
\item $(n-w)^2 - (n-w) \geq \sigma (n-1)$ where $0 \leq \sigma \leq k-1$ and $\sigma \equiv n-2w \Mod k$.
\item If $n$ is odd and $k = 2$, then $w$ is a square.
\end{enumerate}
\end{theorem}

\begin{theorem}[cf. Theorem 5.1, \cite{deLauneyExistence}]\label{deLNE}
Suppose there exists a $\CGW(n,w;k)$ with $n$ odd and $k$ a prime. Suppose that $m \not\equiv 0 \Mod k$ is an integer dividing the square free part of $w$. Then the order of $m$ modulo $k$ is odd.
\end{theorem}

Adhering to the notation of Theorem \ref{deLNE}, if $G = \langle \zeta_{k} \rangle$ for prime $k>2$, then necessarily $p = k$. If, for example, $n = w = 15$ and $p = 5$, then $n$ is square free, so the existence of a $\BH(n,5)$ requires that $m=3$ is of odd order modulo $5$, which is false. In the same paper, the possibility of a $\CGW(19,10;5)$ is eliminated.

These Theorems are effective for eliminating the possibility of finding elements of $\CGW(n,w;k)$ when $k$ is an odd prime, but it generally fails to extend to composite $k$. However, for Butson matrices, i.e., matrices of full weight, these results were partially extended by Winterhof \cite{Winterhof} using number theoretic techniques and the restriction of Lemma \ref{lem:norms} to include certain cases where $k = p^{r}$ or $k=2p^{r}$ for a prime $p \equiv 3 \Mod 4$. The main result is the following.

\begin{theorem}[Theorem 5, \cite{Winterhof}]\label{thm:Wint}
Let $k = p^{r}$ be a prime power where $p \equiv 3 \Mod 4$, and let $n = p^{\ell}a^2m$ be odd where $p$ does not divide $m$, $m$ is square free. Then there is no $\BH(n,k)$ or $\BH(n,2k)$ if there is a prime $q \, \mid \, m$ such that $q$ is a non-quadratic residue modulo $p$.
\end{theorem}

Of the various implications of Winterhof's Theorem, perhaps the most frequently cited is a restriction eliminating the existence of a $\BH(3^{\ell}p^{d},6)$ where $p \equiv 5 \Mod 6$ is prime and $d$ is odd, which include cases such as $\BH(5,6)$ and $\BH(15,6)$.
We can similarly use Lemma \ref{lem:norms} to eliminate existence of CGWs. The following is essentially the same argument as the proof of \cite[Corollary 1.4.6]{FerencThesis}.

\begin{proposition}
There is no $\CGW(n,w;6)$ when $n$ is odd and $w \equiv 2 \mod 3$.
\end{proposition}

\begin{proof}
Suppose $W \in \CGW(n,w;6)$. Then $|\mathrm{det}(W)|^2 = w^n$. Since any element of $\mathcal{U}_{6}$ can be written in the form $a + b\zeta_{3}$ for integers $a$ and $b$, it follows that there are integers $a$ and $b$ such that
\[
w^n = |\mathrm{det}(W)|^2 = |a + b\zeta_{3}|^2 = a^2 + b^2 - ab.
\]
It is not possible that $a^{2} + b^{2} - ab \equiv 2 \mod 3$, and so it cannot be that $n$ is odd and $w \equiv 2 \mod 3$.
\end{proof}

This result motivated the following Propositions. The proofs are similar, but none are omitted as there are subtle differences.

\begin{proposition}
There is no $\CGW(n,w;6)$ when $n$ is odd and $w \equiv 2 \mod 4$.
\end{proposition}

\begin{proof}
As before, if such a matrix exists we require that there are integers $a$ and $b$ such that
\[
w^n = |\mathrm{det}(W)|^2 = |a + b\zeta_{3}|^2 = a^2 + b^2 - ab.
\]
First observe that $a^2 + b^2 - ab \equiv 0 \mod 4$ only if both $a$ and $b$ are even. Let $w = 2m$ for some odd $m$. We show that there is no solution to
\[
(2m)^{n} = a^2 + b^2 - ab,
\]
or equivalently,
\[
(2m)^{n} - ab = (a-b)^{2}.
\]
Now let $2^t$ be the largest power of $2$ dividing both $a$ and $b$. Then we require a solution to
\begin{equation}\label{eq:2mod4}
\frac{(2m)^{n}}{2^{2t}} - \frac{ab}{2^{2t}} = \frac{(a-b)^{2}}{2^{2t}}.
\end{equation}
We split the remainder of the proof into two cases.

First suppose that one of $a$ or $b$ is divisible by $2^{t+1}$, but not both. Then $\frac{a-b}{2^{t}}$ is odd, and so the right hand side of Equation \eqref{eq:2mod4} is odd. However, on the left hand side, the term $\frac{ab}{2^{2t}}$ is even, and the term $\frac{(2m)^{n}}{2^{2t}}$ is either an even integer, or not an integer. Thus Equation \eqref{eq:2mod4} is not satisfied.

Next suppose that neither $a$ nor $b$ are divisible by $2^{t+1}$. Then $\frac{a-b}{2^{t}}$ is even, and so the right hand side of Equation \eqref{eq:2mod4} is even. However, on the left hand side, the term $\frac{ab}{2^{2t}}$ is odd, and the term $\frac{(2m)^{n}}{2^{2t}}$ is again either an even integer, or not an integer. Thus Equation \eqref{eq:2mod4} is again not satisfied. This proves the claim.
\end{proof}

\begin{proposition}
There is no $\CGW(n,w;6)$ when $n$ is odd and $w \equiv 6 \mod 9$.
\end{proposition}

\begin{proof}
As before, if such a matrix exists we require that there are integers $a$ and $b$ such that
\[
w^n = |\mathrm{det}(W)|^2 = |a + b\zeta_{3}|^2 = a^2 + b^2 - ab.
\]
Let $w = 3m$ for some $m \equiv 2 \Mod 3$. This time we show that there is no solution to
\[
(3m)^{n} - ab = (a-b)^{2}.
\]
Let $3^t$ be the largest power of $3$ dividing both $a$ and $b$. Then we require a solution to
\[
\frac{(3m)^{n}}{3^{2t}} - \frac{ab}{3^{2t}} = \frac{(a-b)^{2}}{3^{2t}}.
\]
We split the remainder of the proof into two cases.

First suppose that one of $a$ or $b$ is divisible by $3^{t+1}$, but not both. Then $\frac{a-b}{3^{t}}$ is not a multiple of $3$. However, if $n > 2t$ then both terms on the left hand side are multiples of $3$, and if $n < 2t$ the left hand side is not an integer, so the equation is not satisfied.

Next suppose that neither $a$ nor $b$ are divisible by $3^{t+1}$. In this case, the term $\frac{(a-b)^{2}}{3^{2t}} \equiv 1 \Mod 3$. Thus a solution is only possible if $n > 2t$ and if $\frac{ab}{3^{2t}} \equiv 2 \Mod 3$. This implies that $a = 3^{t}a'$ and $b = 3^{t}b'$ where, without loss of generality, $a' \equiv 1 \Mod 3$ and $b' \equiv 2 \Mod 3$. Now consider the equivalent expression
\[
(3m)^{n} = (a+b)^{2} - 3ab.
\]
In this expression $(a+b)^{2} = (3^{t}(a'+b'))^{2}$ is divisible by $3^{2t+2}$, but the highest power of $3$ dividing $3ab$ is $3^{2t+1}$. So if
\[
\frac{(3m)^{n}}{3^{2t+1}} = \frac{(a+b)^{2}}{3^{2t+1}} - \frac{3ab}{3^{2t+1}},
\]
then the right hand side is congruent to $1 \Mod 3$. However, either $n > 2t+1$ and $\frac{(3m)^{n}}{3^{2t+1}} \equiv 0 \Mod 3$, or $n = 2t+1$ and $\frac{(3m)^{n}}{3^{2t+1}} = m^{2t+1} \equiv 2 \Mod 3$. Thus no solution exists.
\end{proof}

Finally, the following generalizes a well known non-existence result for Butson Hadamard matrices.

\begin{proposition}\label{prop:genNoCGW}
Let $p \equiv 2 \Mod 3$ be a prime and let the squarefree part of $w$ be divisible by $p$.
Then there is no $\CGW(n,w;6)$ when $n$ is odd.
\end{proposition}

\begin{proof}
Let $w = p^{r}m$ for odd $r$ and $m \not\equiv 0 \Mod p$.
The first part is similar to the proofs of the previous Propositions. We show that there is no solution to
\[
(p^{r}m)^{n} - ab = (a-b)^{2}.
\]
Let $p^t$ be the largest power of $p$ dividing both $a$ and $b$. Then we require a solution to
\[
\frac{(p^{r}m)^{n}}{p^{2t}} - \frac{ab}{p^{2t}} = \frac{(a-b)^{2}}{p^{2t}}.
\]

Suppose that one of $a$ or $b$ is divisible by $p^{t+1}$, but not both. Then $\frac{a-b}{p^{t}}$ is not a multiple of $p$. However, if $n > 2t$ then both terms on the left hand side are multiples of $p$, and if $n < 2t$ the left hand side is not an integer, so the equation is not satisfied.

Next suppose that neither $a$ nor $b$ are divisible by $p^{t+1}$. On the left hand side, $\frac{(p^{r}m)^{n}}{p^{2t}}$ is an integer only if it is a multiple of $p$, and $\frac{ab}{p^{2t}}$ is not a multiple of $p$. Letting $a = p^{t}a'$ and $b = p^{t}b'$ for $a',b' \not\equiv 0 \Mod p$, a solution can only exist if
\[
-a'b' \equiv (a'-b')^{2} \Mod p.
\]
Rearranging, this implies that
\[
a'^{2} + b'^{2} \equiv a'b' \Mod p.
\]
Multiplying by $(a'b')^{-1}$, and letting $x = a'(b'^{-1})$, this expression reduces to
\[
x^2 - x + 1 \equiv 0 \Mod p.
\]
Should a solution to this expression exist, we would find that $(x-1)^4 \equiv x-1 \Mod p$, so either $x-1 \equiv 1 \Mod p$, or $(x-1)$ is a element of multiplicative order $3$ in $\Z_{p}$. The former implies that $x \equiv 2 \Mod p$ which contradicts $x^2 - x + 1 \equiv 0 \Mod p$ for all primes $p \neq 3$, so we must have the latter. However, if $p \equiv 2 \Mod 3$, then $3$ does not divide $p-1$, so there is no element of multiplicative order $3$ in $\Z_{p}$, so there can be no solution.
\end{proof}

\begin{remark}
Letting $w=n$ and assuming that $n$ is odd, Proposition \ref{prop:genNoCGW} recovers the known condition that a $\BH(n,6)$ exists only if the squarefree part of $n$ is not divisible by a prime $p \equiv 5 \Mod 6$, which is a special case of Theorem \ref{thm:Wint}.
\end{remark}

\subsection{Block designs and the lifting problem}

When the results already outlined in this section are insufficient, some cases need to be investigated individually. For this purpose, particularly when $k$ is prime, it is often easier to consider whether or not the support matrix of a CGW, should it exist, can meet some necessary conditions. In small cases, we can use a well known restriction on the existence of block designs. First we need a definition. Let $n$, $w$ and $\lambda$ be integers where $n > w > \lambda \geq 0$. Let $X$ be a set of size $n$. A \emph{Symmetric balanced incomplete block design} $\mathrm{SBIBD}(n,w,\lambda)$ is a set of $n$ subsets of $X$ of size $w$, called \emph{blocks} such that each unordered pair of distinct elements of $X$ are contained in exactly $\lambda$ blocks. If $A$ is the incidence matrix of the $\mathrm{SBIBD}(n,w,\lambda)$, then
\[
AA^{\top} = wI_{n} + \lambda(J_{n} - I_{n}),
\]
where $J_{n}$ denotes the $n \times n$ matrix of all ones. It is a well known necessary condition that a $\mathrm{SBIBD}(n,w,\lambda)$ exists only if
\begin{equation}\label{eq:SBIBD}
\lambda(n-1) = w(w-1).
\end{equation}

This condition will be useful for eliminating the possibility of a $\CGW(n,w;k)$ for certain small parameters. A reason for this is that we can sometimes observe that a $\CGW(n,w;k)$ can only exist if the support matrix is the incidence matrix of some $\mathrm{SBIBD}(n,w,\lambda)$. For example, it can be shown that if a $\CGW(11,5;4)$ exists, then it's support must be a $\mathrm{SBIBD}(11,5,2)$. Such a design exists, but it is unique, so we only need to check if it is possible that the incidence matrix of this design can support a $\CGW(11,5;4)$, which we can do by hand. This is one way to eliminate the existence of a $\CGW(11,5;4)$. It is an example of the following problem.

\begin{problem}[The lifting problem]
{\upshape Given a $(0,1)$-matrix $S$, does $S$ lift to a $\CGW(n,w;k)$?}
\end{problem}

In several cases, non-existence of a $\CGW(n,w;k)$ is verified by showing that there does not exist a $(0,1)$-matrix $S$ that lifts to a $\CGW(n,w;k)$ for the given parameters. Completing the unfilled entries in the existence tables in Appendix \ref{App:Tables} may require solving the lifting problem, as potential support matrices exist in those cases.

\subsection{Sporadic non-existence conditions}

One of our aims will be to settle the question of existence for all many orders $1 \leq n \leq 15$ and weights $1 \leq w \leq n$ as we can for small $k$; see Section \ref{section:Existence} and the Tables in Appendix \ref{App:Tables}.
Non-existence is mostly determined by the results already described in this section, but occasionally some more specialized results are required. Existence is in most cases given as a result of one of the constructions of Section \ref{section:Construct}. In some cases we can prove non-existence for certain parameters individually, which often reduces to determining if a support matrix can exist, and if so, trying to solve the lifting problem. This section is not intended to be comprehensive, but to demonstrate the kind of methods that can be implemented at small orders. We give some examples here.

\begin{proposition}
There exists a $\CGW(n,4;3)$ if and only if $n \equiv 0 \Mod 5$.
\end{proposition}

\begin{proof}
Let $W \in \CGW(n,4;3)$ and let $S$ be the support matrix of $W$. Then the dot product of any two distinct rows must be either $0$ or $3$. Any pair of the four rows that contain a $1$ in column $1$ must therefore share a $1$ in exactly two other columns, and so up to permutation equivalence, these rows are of the form
\[
{\footnotesize\renewcommand{\arraycolsep}{.1cm}
\left[\begin{array}{ccccc|ccc}
1 & 1 & 1 & 1 & 0 & 0 & \cdots & 0 \\
1 & 1 & 1 & 0 & 1 & 0 & \cdots & 0 \\
1 & 1 & 0 & 1 & 1 & 0 & \cdots & 0 \\
1 & 0 & 1 & 1 & 1 & 0 & \cdots & 0  \end{array}\right].}
\]
In this configuration, columns $2,3,4,5$ share a $1$ in exactly $2$ rows, and have only one further $1$ remaining in the column and so we can immediately deduce a fifth row of the matrix which up to equivalence takes the form
\[
{\footnotesize\renewcommand{\arraycolsep}{.1cm}
\left[\begin{array}{ccccc|ccc}
1 & 1 & 1 & 1 & 0 & 0 & \cdots & 0 \\
1 & 1 & 1 & 0 & 1 & 0 & \cdots & 0 \\
1 & 1 & 0 & 1 & 1 & 0 & \cdots & 0 \\
1 & 0 & 1 & 1 & 1 & 0 & \cdots & 0 \\
0 & 1 & 1 & 1 & 1 & 0 & \cdots & 0  \end{array}\right] = [C ~ \mid ~ 0_{5,n-5}].}
\]
Proceeding in the same way, we find that $S$ must be permutation equivalent to a block diagonal matrix with the $5 \times 5$ matrix $C$ in block on the diagonal. The claim that $n \equiv 5$ follows immediately. To see that $n \equiv 5$ is sufficient, let $C$ be the support of the $\CGW(5,4;3)$ of Example \ref{ex:Berman}.
\end{proof}

\begin{proposition}
There is no matrix in $\CGW(10,6;3)$.
\end{proposition}
\begin{proof}

Should such a matrix $W$ exist, then it must have exactly $4$ zeros in each row and column, and in any distinct row/column they must share the entry zero in either $1$ or $4$ positions. Suppose first that no two rows share a zero in $4$ positions, and so each pair share a zero in exactly one position. Then if $S$ is the support matrix of $W$, the matrix $J_{n} - S$ should be the incidence matrix of a $\mathrm{SBIBD}(10,4,1)$. However, these parameters contradict Equation \eqref{eq:SBIBD}.

Next suppose that some pair of distinct rows have their zeros in the same four columns. Then because these columns now share a zero in at least $2$ positions, they must also do so in $4$. As a result, up to equivalence, the support matrix must have $4$ rows such that it takes the form
\[
{\footnotesize\renewcommand{\arraycolsep}{.1cm}
\left[\begin{array}{cccccc|ccc}
1 & 1 & 1 & 1 & 1 & 1 & 0 & \cdots & 0 \\
1 & 1 & 1 & 1 & 1 & 1 & 0 & \cdots & 0 \\
1 & 1 & 1 & 1 & 1 & 1 & 0 & \cdots & 0 \\
1 & 1 & 1 & 1 & 1 & 1 & 0 & \cdots & 0  \end{array}\right].}
\]
Now, in any distinct pair of the first six columns, the entries equal to $1$ are in $4$ common rows, and so the remaining two $1$s must also be in common rows. Thus, up to equivalence, the two subsequent rows take the form
\[
{\footnotesize\renewcommand{\arraycolsep}{.1cm}
\left[\begin{array}{cccccc|ccc}
1 & 1 & 1 & 1 & 1 & 1 & \ast & \cdots & \ast \\
1 & 1 & 1 & 1 & 1 & 1 & \ast & \cdots & \ast  \end{array}\right].}
\]
In order that the rows have weight $6$, the entries marked $\ast$ must be zero, but then the corresponding columns would have weight $ \leq 4$, and so we have a contradiction.
\end{proof}

\begin{proposition}\label{prop:10-7-4}
There is no $\CGW(10,7;4)$.
\end{proposition}

\begin{proof}
Suppose that $W \in \CGW(10,7;4)$ and let $S$ be the support matrix of $W$. The positions of the ones in any two rows intersect in either $4$ or $6$ places. If positions of the ones in all pairs of distinct rows intersected in exactly $4$ places then $S$ would describe a $(10,7,4)$-design, which is forbidden by Equation \eqref{eq:SBIBD}. So at least two rows share ones in $6$ positions, and up to equivalence the first two rows of $S$ are
\[
{\footnotesize\renewcommand{\arraycolsep}{.1cm}
\left[\begin{array}{cccccccccc}
0 & 1 & 1 & 1 & 1 & 1 & 1 & 1 & 0 & 0 \\
1 & 0 & 1 & 1 & 1 & 1 & 1 & 1 & 0 & 0  \end{array}\right].}
\]
Now in any subsequent row, if there are an even number of ones in positions $3$ to $8$, then both entries in the positions $1$ and $2$ are zero. If there are an odd number of ones in positions $3$ to $8$, then both entries in the positions $1$ and $2$ are one. It follows that, up to equivalence, $S$ takes the form
\[
{\footnotesize\renewcommand{\arraycolsep}{.1cm}
\left[\begin{array}{cc|cccccccc}
0 & 1 & 1 & 1 & 1 & 1 & 1 & 1 & 0 & 0 \\
1 & 0 & 1 & 1 & 1 & 1 & 1 & 1 & 0 & 0 \\ \hline
1 & 1 & & & & & & & & \\
1 & 1 & & & & & & & & \\
1 & 1 & & & & & & & & \\
1 & 1 & & & & & & & & \\
1 & 1 & & & & & & & & \\
1 & 1 & & & & & & & & \\
0 & 0 & & & & & & & & \\
0 & 0 & & & & & & & & \end{array}\right].}
\]
Now, in rows $3$ to $8$, there are an odd number of ones in columns $3$ to $8$, and so an even number of ones in columns $9$ and $10$. Since zeros in columns $9$ and $10$ already meet in rows $1$ and $2$, this cannot happen again and so both entries in rows $3$ to $8$ must equal $1$. Completing rows $9$ and $10$ is similar, and we find $S$ is of the form
\[
{\footnotesize\renewcommand{\arraycolsep}{.1cm}
\left[\begin{array}{cc|cccccc|cc}
0 & 1 & 1 & 1 & 1 & 1 & 1 & 1 & 0 & 0 \\
1 & 0 & 1 & 1 & 1 & 1 & 1 & 1 & 0 & 0 \\ \hline
1 & 1 & & & & & & & 1 & 1 \\
1 & 1 & & & & & & & 1 & 1 \\
1 & 1 & & & & & & & 1 & 1 \\
1 & 1 & & & & & & & 1 & 1 \\
1 & 1 & & & & & & & 1 & 1 \\
1 & 1 & & & & & & & 1 & 1 \\ \hline
0 & 0 & 1 & 1 & 1 & 1 & 1 & 1 & 1 & 0 \\
0 & 0 & 1 & 1 & 1 & 1 & 1 & 1 & 0 & 1 \end{array}\right].}
\]

Now consider the $6 \times 6$ submatrix in the centre, which must have exactly three entries equal to one in each row and column, and the ones in any pair of rows must meet in either $0$ or $2$ positions. If the ones in any two rows meet in zero positions, it is impossible to complete a third row that meets each of these two in $0$ or $2$ positions, so they must all meet in exactly two positions. However this would imply that the submatrix in the centre describes a $(6,3,2)$-design, which is also forbidden by Equation \eqref{eq:SBIBD}.
\end{proof}

We give one more example of this kind of argument.

\begin{proposition}
There is no $\CGW(11,5;4)$.
\end{proposition}

\begin{proof}
Using similar arguments to those of Proposition \ref{prop:10-7-4}, it can be shown that the support matrix of such a matrix must be the incidence matrix of a $(11,5,2)$-design. There is, up to equivalence, exactly one such design (see, e.g., \cite{HandbookCD}). Thus we must be able to solve the lifting problem for this particular support matrix if a $\CGW(11,5;4)$ is to exist. However, it is not difficult to verify that this is impossible. We omit details for brevity.
\end{proof}

\section{Constructions}\label{section:Construct}

In this section we outline the best known constructions of CGWs. We begin with direct constructions and infinite families, including a summarization of the best known work on the topic by Berman and Seberry and Whiteman. We then consider various recursive constructions, including standard direct sum or tensor product constructions, and more general recursive constructions such as the powerful method of weaving introduced by Craigen. We begin with a method strongly influenced by a familiar construction of conference matrices due to Paley.

\subsection{Generalized Paley}

The most famous constructions of an infinite family of Hadamard matrices are due to Paley. There are two constructions yielding what are now known as the type I and type II Paley Hadamard matrices. Both constructions are built on circulant cores, obtained by applying the quadratic character to the elements of a finite field $\mathbb{F}_{q}$. The next construction we introduce is not strictly a generalization of Paley's construction of the circulant core, but bears a strong enough resemblance that we refer to this as a generalized Paley construction. Let $p$ and $q$ be primes, with $q \equiv 1 \mod p$. Let $x$ be a multiplicative generator of the non-zero elements of $\mathbb{Z}_{q}$. Consider the map $\phi : \mathbb{Z}_{q} \rightarrow \langle \zeta_{p} \rangle \cup \{0\}$ defined by setting $\phi(x^{j}) = \zeta_{p}^{j}$ for all $1 \leq j \leq q-1$, and setting $\phi(0) = 0$. Then $\phi$ has the following two properties:
\begin{itemize}
\item $\phi(xy) = \phi(x)\phi(y)$ for all $x,y \in \mathbb{Z}_{q}$; and
\item $\phi(x^{*}) = \phi(x)^{*}$ for all $x \in \mathbb{Z}_{q}$.
\end{itemize}

\begin{lemma}
Let $C = \mathrm{circ}([\phi(x) \; : \; 0 \leq x \leq q-1])$. Then $CC^{*} = qI_{q} - J_{q}$.
\end{lemma}

\begin{proof}
Observe that each row has exactly $q-1$ non-zero entries, so the diagonal entries of $CC^{*}$ are clearly as claimed. It remains to show that the Hermitian inner product of any two distinct rows $r_{i}$ and $r_{j}$ of $C$ is $-1$. Since $C$ is circulant, this inner product is
\[
\langle r_{i},r_{j} \rangle = \sum_{x \in \mathbb{Z}_{q}}\phi(x)\phi(x-s)^{*}
\]
for some $s \neq 0$. Using the properties of $\phi$ we observe that
\[
\phi(x)\phi(x-s)^{*} = \phi(x(x-s)^{*}).
\]
Now, $x(x-s)^{*} = 0$ if and only if $x = 0,s$. If $x,y \not\in\{0,s\}$, then
\begin{align*}
x(x-s)^{*} &= y(y-s)^{*}\\
\Leftrightarrow x(y-s) &= y(x-s) \\
\Leftrightarrow xs &= ys \\
\Leftrightarrow x &= y.
\end{align*}
Further, $x(x-s)^{*} = 1$ only if $s = 0$.
It follows that when $s \neq 0$, the multiset $\{x(x-s)^{*} \; : \; x \in \mathbb{Z}_{q}\setminus \{0,s\}\} = \{2,3,\ldots,q-1\}$.
Consequently, by Theorem \ref{LamLeungVan}, $\sum_{x \in \mathbb{Z}_{q}}\phi(x)\phi(x-s)^{*} = -1$, as required.
\end{proof}

The following is now immediate.

\begin{theorem}
Let $C = \mathrm{circ}([\phi(x) \; : \; 0 \leq x \leq q-1])$. Then the matrix
\[
W = {\footnotesize\renewcommand{\arraycolsep}{.1cm}
\left[\begin{array}{c|c}
0 & {\bf 1} \\ \hline
{\bf 1}^{\top} & C \end{array}\right]},
\]
is a $\mathrm{CGW}(q+1,q;p)$.
\end{theorem}

\subsection{Berman's constructions}

The earliest constructions of CGWs that we know of are due to Berman \cite{Berman77,Berman78}. We have only been able to obtain a copy of the more recent paper \cite{Berman78}, which claims to generalize the constructions in \cite{Berman77} which are limited to real weighing matrices. Families are constructed using connections to finite geometry.

Let $p$, $n$, and $t$ be positive integers with $p$ a prime. Let $F$ be the finite field $\F_{p^{n}}$ and let $P'$ be the set of all points in the affine space $F^{t}$, excluding the origin ${\bf 0} = (0,0,\ldots,0)$. Let $H'$ denote the set of hyperplanes of $F^{t}$ that do not include ${\bf 0}$. Hence $|P'| = |H'| = p^{tn}-1$. A hyperplane in $F^{t}$ can be described by a linear equation of the form
\[
u_{1}x_{1} + \cdots + u_{t}x_{t} = b
\]
for an arbitrary constant $b \in F$. In order for this construction to work, we cannot choose $b=0$. Adhering to the choice of Berman in \cite{Berman78}, we let $b = 1$. Thus
every hyperplane can be described by a $t$-tuple ${\bf u} = (u_{1},\ldots,u_{t})$ where $u_{i} \in F$, where each ${\bf u} \in H'$ satisfies a linear equation
\[
u_{1}x_{1} + \cdots + u_{t}x_{t} = 1
\]
where at least one $u_{j} \neq 0$. Letting ${\bf x}$ be a point in $P'$, it follows that this equation can be written as ${\bf u}{\bf x}^{\top} = 1$. We say that the point ${\bf x}$ is on the hyperplane ${\bf u}$ or that ${\bf u}$ contains the point ${\bf x}$ and write ${\bf x} \in {\bf u}$ if ${\bf x}$ and ${\bf u}$ satisfy this equation. It follows that a point ${\bf x} \in P'$ is on $p^{(t-1)n}$ hyperplanes of $H'$ and a hyperplane ${\bf u} \in H'$ contains $p^{(t-1)n}$ points of $P'$.

A collineation $\phi$ is a transformation of $F^{t}$ preserving collinearity; the order of $\phi$ is the smallest $r$ such that $\phi^{r}$ is the identity transformation. The map $\phi_{\lambda} : {\bf x} \mapsto \lambda {\bf x}$ for $\lambda \in F \setminus \{0\}$ is a collineation of order $r_{\lambda}$ which maps the hyperplane ${\bf u}$ onto the hyperplane $\lambda^{-1}{\bf u}$. Writing $[{\bf x}] = \{\phi_{\lambda}^{j} {\bf x} \; : \; j = 0,\ldots,r_{\lambda}-1\}$ for ${\bf x} \in P'$, we observe that ${\bf y} \in [{\bf x}]$ if and only if ${\bf x} \in [{\bf y}]$. It follows that $P' = [{\bf x}^{(1)}] \cup [{\bf x}^{(2)}] \cup \cdots \cup [{\bf x}^{(m)}]$ is a partition into $m$ classes, where $mr_{\lambda} = p^{tn}-1$. Similarly, we have the partition $H' = [{\bf u}^{(1)}] \cup [{\bf u}^{(2)}] \cup \cdots \cup [{\bf u}^{(m)}]$. If ${\bf x}^{(j)}$ is a point of $\phi_{\lambda}^{\ell}{\bf u}^{(i)}$, then
\[
1 = \lambda^{-\ell}{\bf u}^{(i)}({\bf x}^{(j)})^{\top} = \lambda^{-\ell-k}{\bf u}^{(i)}(\lambda^{k}{\bf x}^{(j)})^{\top}
\]
for any $0 \leq k \leq r_{\lambda} - 1$, and so $\phi_{\lambda}^{k}{\bf x}^{(j)}$ is a point of $\phi_{\lambda}^{\ell+k}{\bf u}^{(i)}$. It follows that if a point ${\bf x}^{(j)}$ lies on any hyperplane in $[{\bf u}^{(i)}]$, then each point in $[{\bf x}^{(j)}]$ lies on exactly one hyperplane in $[{\bf u}^{(i)}]$.
As such, if points of $[{\bf x}^{(j)}]$ are on hyperplanes of $[{\bf u}^{(i)}]$, we write $[{\bf x}^{(j)}] \in [{\bf u}^{(i)}]$.

Now let $d>1$ be any divisor of $r_{\lambda}$, and let $v({\bf u}^{(i)},{\bf x}^{(j)})$ be the unique integer $h$ such that $\phi_{\lambda}^{h}{\bf x}^{(j)}$ is a point of ${\bf u}^{(i)}$. Finally, let $A$ be the $m \times m$ matrix $(a_{ij})$ defined by
\[
a_{ij} = \begin{cases}
\zeta_{d}^{v({\bf u}^{(i)},{\bf x}^{(j)})} ~~ &\text{if} ~ [{\bf x}^{(j)}] \in [{\bf u}^{(i)}] \\
0 ~~ &\text{otherwise.}\end{cases}
\]

Assuming all of the notation of this section, we have the following.

\begin{theorem}[cf.{~\cite[Theorem 2.2]{Berman78}}]
The matrix $A$ is an element of $\CGW((p^{tn}-1)/r_{\lambda},p^{(t-1)n};d)$.
\end{theorem}

\begin{proof}
The parameters of $A$ are all clear from its construction. It remains to show that $A$ is orthogonal, i.e., that
\[
Q = \sum a_{ij}\overline{a_{kj}} = 0
\]
for all $i \neq j$. The $i^{\rm th}$ and $k^{\rm th}$ rows of $A$ correspond to hyperplane classes. If $[{\bf u}^{(i)}]$ and $[{\bf u}^{(k)}]$ are parallel, then they have no points in common, and it follows that the sum $Q$ contains no non-zero terms. Suppose then that $[{\bf u}^{(i)}]$ and $[{\bf u}^{(k)}]$ do intersect, and so ${\bf u}^{(i)}$ intersects each of the hyperplanes $\phi_{\lambda}^{h}{\bf u}^{(k)}$, $h = 0,1,\ldots,r_{\lambda}-1$, in $p^{(t-2)n}$ points. Thus the sum $Q$ contains $r_{\lambda}p^{(t-2)n}$ non-zero terms. For any point $\phi_{\lambda}^{\ell}{\bf x}^{(j)}$ on each of ${\bf u}^{(i)}$ and $\phi_{\lambda}^{h}{\bf u}^{(k)}$, we have
\[
{\bf u}^{(i)}(\lambda^{\ell}{\bf x}^{(j)})^{\top} = 1
\]
and
\[
(\lambda^{h}{\bf u}^{(k)})(\lambda^{\ell}{\bf x}^{(j)})^{\top} = {\bf u}^{(k)}(\lambda^{\ell-h}{\bf x}^{(j)})^{\top} = 1
\]
so that $v({\bf u}^{(i)},{\bf x}^{(j)}) = \ell$ and $v({\bf u}^{(k)},{\bf x}^{(j)}) = \ell-h$. Consequently, $a_{ij} = \zeta_{d}^{\ell}$ and $a_{kj} = \zeta_{d}^{\ell-h}$ and so $a_{ij}\overline{a_{kj}} = \zeta_{d}^{h}$. Thus for all $h = 0,1,\ldots,r_{\lambda}-1$, there are $p^{(t-2)n}$ terms of $Q$ which have the value $\zeta_{d}^{h}$, and so
\[
Q = p^{(t-2)n}(1+\zeta_{d}+\cdots + \zeta_{d}^{r_{\lambda}-1}) = 0.
\]
\end{proof}

\begin{corollary}[cf.{~\cite[Corollary 2.3]{Berman78}}]\label{cor:BermanPar}
Let $p$, $n$, $t$, $d$ and $r$ be any positive integers such that $p$ is prime, $d \mid r$, and $r \mid (p^{n}-1)$. Then there exists a matrix $W$ in $\CGW((p^{tn}-1)/r,p^{(t-1)n};d)$.
\end{corollary}

\begin{example}\label{ex:Berman}
{\upshape
Letting $p = n = 2$, then for any choice of $t>1$ we can let $d=r=3$, and build a matrix in $\CGW((2^{2t}-1)/3,2^{2t-2};3)$. When $t=2$, we get a matrix
\[
W \equiv {\footnotesize\renewcommand{\arraycolsep}{.1cm}
\left[\begin{array}{ccccc}
0 & 1 & 1 & 1 & 1 \\
1 & 0 & 1 & \zeta_{3} & \zeta_{3}^2 \\
1 & 1 & 0 & \zeta_{3}^2 & \zeta_{3} \\
1 & \zeta_{3} & \zeta_{3}^2 & 0 & 1 \\
1 & \zeta_{3}^2 & \zeta_{3} & 1 & 0 \end{array}\right]}.
\]
}
\end{example}

\begin{remark}
Berman also provides constructions of $\zeta$-circulant CGWs in \cite{Berman78}. They involve some specialised edits to the construction above, but do not construct matrices with parameters distinct from those facilitated for by Corollary \ref{cor:BermanPar}. As such, in the interest of brevity we just refer the reader to the original paper for more details.
\end{remark}

\subsection{Complementary sequences}

This subsection summarises the work of \cite{GenGPpaper} as it pertains to CGWs. Perhaps the best known example of complementary sequences are Golay pairs \cite{Golay49}. These are pairs of $\{ \pm 1 \}$-sequences $(a,b)$ of length $v$ such that
\[
\sum_{j=0}^{v-1-s}a_{j}a_{j+s} + b_{j}b_{j+s} = 0
\]
for all $1 \leq s \leq v-1$. This equation says that the aperiodic autocorrelation of the sequences $a$ and $b$ sum to zero for all possible shifts $s$. The existence of Golay pairs is known when the length of the sequences is $v = 2^{x}10^{y}26^{z}$ for $x,y,z \geq 0$, but not for any other values. This motivated the generalization to complementary sequences according to periodic autocorrelation functions. We describe a very general extension of the idea here, as it pertains to the construction of CGWs.

In this section, we identify sequences with row vectors, for the purposes of describing certain operations with matrix multiplication. For any $\alpha \in \mathcal{U}_{k}$, define the $\alpha$-circulant matrix
\[
C_{\alpha} = {\footnotesize\renewcommand{\arraycolsep}{.1cm}
\left[\begin{array}{cccccc}
0 & 1 & 0 & \cdots & 0 & 0 \\
0 & 0 & 1 &  & 0 & 0 \\
0 & 0 & 0 &  & 0 & 0 \\
\vdots &  &  & \ddots &  & \vdots \\
0 & 0 & 0 &  & 0 & 1 \\
\alpha & 0 & 0 & \cdots & 0 & 0\end{array}\right]}.
\]
The $\alpha$-\emph{phased periodic autocorrelation function} of a $\mathcal{U}_{k}$-sequence $a$ of length $v$ and shift $s$ to be
\[
\PPAF_{\alpha,s}(a) = a (aC_{\alpha}^s)^{\ast}.
\]
Let $(a,b)$ be a pair of $\mathcal{U}_{k}$-sequences. Let $w_{a}$ denote the weight of a sequence $a$, i.e., the number of non-zero entries in $a$.  We say $w = w_{a}+w_{b}$ is the weight of a pair $(a,b)$.
A pair of sequences $(a,b)$ is a \emph{weighted $\alpha$-phased periodic Golay pair} ($\WPPGP(\mathcal{U}_{k},v,\alpha,w)$) if
\[
\PPAF_{\alpha,s}(a) + \PPAF_{\alpha,s}(b) = 0.
\]
for all $1 \leq s \leq v-1$.

For some $\alpha \in \mathcal{U}_{k}$, let $A$ and $B$ be the $\alpha$-circulant matrices with first row $a$ and $b$ respectively, that is $A_{i+1} = A_{i}C_{\alpha}$ and $B_{i+1} = B_{i}C_{\alpha}$ for all $2 \leq i \leq v$. When $a$ and $b$ are complementary, we construct a matrix with pairwise orthogonal rows as follows.

\begin{theorem}[Theorem 5.1 \cite{GenGPpaper}]\label{thm:WPPGP-construction}
Let $(a,b) \in \WPPGP(\mathcal{U}_{k},v,\alpha,w)$ and define the matrices $A$ and $B$ as above.  If
\[
W = {\small\renewcommand{\arraycolsep}{.1cm}
\left[\begin{array}{cc}
A & B \\
-B^{*} & A^{*}\end{array}\right],}
\]
\noindent then $WW^{*} = wI_{2v}$. That is, $W$ is a $\CGW(2v,w;2k)$ if $k$ is odd, and $W$ is a $\CGW(2v,w;k)$ if $k$ is even. \end{theorem}

The constructions of $\WPPGP(\mathcal{U}_{k},v,\alpha,w)$ that appear most frequently in the literature are limited to the when $k = 2$ or $k = 4$, and when $\alpha  = 1$ or $\alpha = \frac{k}{2}$. Computational methods for searching are useful, but ultimately limited to small length sequences. However, we can take advantage of constructions of aperiodic complementary sequences. In particular, a \emph{ternary Golay pair} of is a pair of $(0, \pm 1)$-sequences $(a,b)$ of length $n$ such that
\[
\sum_{j=0}^{n-1-s}a_{j}a_{j+s} + b_{j}b_{j+s} = 0
\]
for all $1 \leq s \leq n-1$. There is a range of studies of ternary Golay pairs in the mathematics and engineering literature for numerous reasons, we refer the reader to \cite{CraKou01} and \cite{GysinSeb} for more details. The following is a special case of \cite[Theorem 3.5]{GenGPpaper}.

\begin{theorem}\label{thm:intersect}
Let $(a,b)$ be a ternary Golay pair of length $n$ and weight $w$. Then $(a,b) \in \WPPGP(\mathcal{U}_{k},n,\alpha,w)$ for any even $k$, and any $\alpha \in \langle \zeta_{k} \rangle$.
\end{theorem}

As a consequence, given $(a,b)$ we can construct several distinct matrices in $\CGW(2n,w;k)$ that are not equivalent to a $\CGW(2n,w;2)$.

\subsection{Seberry and Whiteman}

For any prime power $q \equiv 1 \mod 8$, Seberry and Whiteman \cite{SeberryWhiteman} present a construction of a matrix in $\CGW(q+1,q;4)$. It is essentially a construction of an element of $\WPPGP(\mathcal{U}_{4},\frac{q+1}{2},1,q)$, although this is not how the construction is described in the paper.

Let $i = \zeta_{4}$ in this section. The method involves constructing two circulant matrices $R$ and $S$ of order $\frac{q+1}{2}$ with all entries in $\{\pm 1,\pm i\}$ except for the diagonal of $R$ which is $0$, and building the order $q+1$ matrix
\begin{equation}\label{SW-matrix}
W \equiv {\footnotesize\renewcommand{\arraycolsep}{.1cm}
\left[\begin{array}{cc}
R & S \\
S^{*} & -R^{*} \end{array}\right]}.
\end{equation}
The sequence of entries of their first rows is obtained by cleverly applying an eighth power character $\chi$ to elements of $\F_{q^{2}}$ in a particular order.

The method is as follows. Let $q \equiv 1 \mod 8$ be a prime power and let $n = \frac{q+1}{2}$. Let $\tau$ be a primitive element of $\F_{q^2}$ and let $\gamma = \tau^n$. For $x \in \F_{q^2} \setminus \{0\}$, let $\mathrm{ind}(x)$ be the least non-negative integer $t$ such that $\tau^{t} = x$ and define $\chi : \F_{q^2} \rightarrow \langle \zeta_{8} \rangle \cup \{0\}$ so that
\begin{equation}\label{chi-eq}
\chi(x) = \begin{cases}
\zeta_{8}^{\mathrm{ind}(x)} ~~ &\text{if} ~ x \neq 0\\
0 ~~ &\text{if} ~ x = 0. \end{cases}
\end{equation}

We can write each element of $\F_{q^2}$ uniquely in the form $\alpha \gamma + \beta$ where $\alpha,\beta \in \F_{q}$. So let $\tau^{j} = \alpha_{j} \gamma + \beta_{j}$ for all $j$, and define the sequences $a$ and $b$ such that $a_{j} = \chi(\alpha_{j})$ and $b_{j} = \chi(\beta_{j})$. These sequences satisfy the following two identities for all $0\leq j \leq q^2-2$;
\begin{equation}\label{ident1}
b_{j+2n} = ib_{j},
\end{equation}
\begin{equation}\label{ident2}
b_{j+n} = ia_{j}.
\end{equation}

The first rows $r$ and $s$ of $R$ and $S$ are chosen to be the subsequences $r = (a_{0},a_{8},\ldots,a_{8(n-1)})$ and $s= (b_{0},b_{8},\ldots,b_{8(n-1)})$ respectively. The matrix of the form in \eqref{SW-matrix} is orthogonal only if the sequences $r$ and $s$ are complementary, i.e., if
\begin{equation}\label{Paut}
\sum_{j=0}^{n-1}a_{8j}\overline{a_{8j+8t}} + b_{8j}\overline{b_{8j+8t}} = 0
\end{equation}
for $1 \leq t \leq n-1$, where the indices are read modulo $8n$.
Appealing to \eqref{ident1} this requirement reduces to
\begin{equation}\label{odd-even-eq}
\sum_{j=0}^{n-1}b_{8j}\overline{b_{8j+8t}} + b_{8j+n}\overline{b_{8j+n+8t}} = 0.
\end{equation}
The identity of \eqref{ident2} implies that $b_{j}\overline{b_{\ell}} = b_{j+2n}\overline{b_{\ell+2n}}$ for all $j,\ell$. Thus we can write the indices modulo $2n = q+1$ in the sum above. Further, because $n$ is odd, the indices $8j$ in the left hand term of the sum of \eqref{odd-even-eq} covers the even integers between $0$ and $q-1$ and the indices $8j+n$ in right hand term covers the odd integers between $1$ and $q$, and we get the equivalent expression
\[
\sum_{j=0}^{n-1}b_{2j}\overline{b_{2j+8t}} + b_{2j+1}\overline{b_{2j+1+8t}} = 0.
\]
More succinctly, we get
\begin{equation}\label{b-eq}
\sum_{j=0}^{q}b_{j}\overline{b_{j+8t}} = 0.
\end{equation}

It remains to show that Equation \eqref{b-eq} holds for all $1 \leq t \neq n-1$. First, observe that $b_{j}\overline{b_{j+\ell}} = \chi(\beta_{j})\overline{\chi(\beta_{j+\ell})}$. Suppose that $\tau^{j} = \alpha_{j}\gamma + \beta_{j}$ and $\tau^{\ell} = \alpha_{\ell}\gamma + \beta_{\ell}$. Then
\begin{align*}
\tau^{j+\ell} &= (\alpha_{j}\gamma + \beta_{j})(\alpha_{\ell}\gamma + \beta_{\ell})\\
&= (\alpha_{j}\beta_{\ell} + \beta_{j}\alpha_{\ell})\gamma + (\alpha_{j}\alpha_{\ell}\gamma^{2} + \beta_{j}\beta_{\ell}).
\end{align*}
and so $\beta_{j+\ell} = \alpha_{j}\alpha_{\ell}\gamma^{2} + \beta_{j}\beta_{\ell}$. It follows that
\[
\chi(\beta_{j})\overline{\chi(\beta_{j+\ell})} = \chi(\beta_{j})\overline{\chi(\alpha_{j}\alpha_{\ell}\gamma^{2} + \beta_{j}\beta_{\ell})}.
\]
Consequently, for any fixed $\ell \neq 0$,
\begin{equation}\label{chibet}
\sum_{j=0}^{q^2-2}b_{j}\overline{b_{j+\ell}} = \sum_{j=0}^{q^2-2}\chi(\beta_{j})\overline{\chi(\alpha_{j}\alpha_{\ell}\gamma^{2} + \beta_{j}\beta_{\ell})}.
\end{equation}
Alternatively,
\begin{align*}
\sum_{j=0}^{q^2-2}b_{j}\overline{b_{j+\ell}} &= \sum_{\alpha,\beta \in \F_{q}}\chi(\beta)\overline{\chi(\alpha\alpha_{\ell}\gamma^{2} + \beta\beta_{\ell})}\\
&= \sum_{\beta \in \F_{q}}\chi(\beta)\sum_{\alpha \in \F_{q}}\overline{\chi(\alpha\alpha_{\ell}\gamma^{2} + \beta\beta_{\ell})}
\end{align*}

where the inner sum is zero whenever $\alpha_{\ell} \neq 0$. Also note 
that the only value of the form $\ell = 8t$ with $0 \leq t \leq n-1$ for which $\alpha_{\ell}=0$ when $8 | (q-1)$ is when $t=0$.  It follows that
\[
\sum_{j=0}^{q^2-2}b_{j}\overline{b_{j+8t}} = 0
\]
for any $1 \leq t \leq n-1$.
Now we note that
\[
\sum_{j=0}^{q^2-2}b_{j}\overline{b_{j+8t}} = \sum_{h=0}^{q-2}\sum_{j=0}^{q}b_{j+h(q+1)}\overline{b_{j+8(t+h(q+1))}}.
\]
Appealing again to \eqref{ident1} we note that the inner sum takes the same value for every $h$ because $q+1 = 2n$, and so letting $h=0$, we conclude that
\[
\sum_{j=0}^{q}b_{j}\overline{b_{j+8t}} = 0.
\]

This proves the following.

\begin{theorem}[Theorem 2, \cite{SeberryWhiteman}]
For any prime power $q \equiv 1 \mod 8$ there exists a matrix in $\CGW(q+1,q;4)$.
\end{theorem}

\begin{example}\label{ex:SW}
{\upshape
Let $q = 9$, so $n = 5$. Let $\tau$ be a primitive element of $\F_{81}$, let $\gamma = \tau^5$, and let $z = \tau^{10}$ so that $z$ is a primitive element of a subfield isomorphic to $\F_{9}$. Then
\begin{align*}
\tau^{0} &= 0\cdot\gamma + z^8 = 1,\\
\tau^{8} &= z^{5}\cdot\gamma + z^7,\\
\tau^{16} &= z^{8}\cdot\gamma + z,\\
\tau^{24} &= z^{8}\cdot\gamma + z^5,\\
\tau^{32} &= z^{5}\cdot\gamma + z^3.
\end{align*}
Adhering to Equation \eqref{chi-eq} we find $r = (0,i,1,1,i)$ and $s = (1,-i,i,i,-i)$. It is simple to verify that the matrix $W$ defined as in Equation \eqref{SW-matrix} is an element of $\CGW(10,9;4)$. 
}
\end{example}

\subsection{Recursive constructions}

The constructions above, in addition to the many constructions known for real weighing, Hadamard and Butson Hadamard matrices provide numerous CGWs at infinitely many orders. However, recursive constructions applied to these matrices are the most effective tool for producing large quantities of CGWs. Tensor product type constructions are the most frequently useful, however we will begin with the simplest of recursive constructions of a direct sum type.

\subsubsection{Direct sum type constructions}

We define the direct sum of an $m \times m$ matrix $A$ and an $n \times n$ matrix $B$ to be
\[
A \oplus B =  {\footnotesize\renewcommand{\arraycolsep}{.1cm}
\left[\begin{array}{cc}
A & 0_{m,n} \\
0_{n,m} & B\end{array}\right].}
\]
The following is immediate.
\begin{proposition}
If $A \in \CGW(m,w;k_{1})$ and $B \in \CGW(n,w;k_{2})$, then $A \oplus B \in \CGW(m+n,w;k)$ where $k = \mathrm{lcm}(k_{1},k_{2})$.
\end{proposition}

The following is also quite straightforward.

\begin{proposition}\label{prop:add1}
Let $A \in \CGW(n,w;k)$. Then the matrix
\[
{\footnotesize\renewcommand{\arraycolsep}{.1cm}
\left[\begin{array}{cc}
A & I_{n} \\
-I_{n} & A^{\ast}\end{array}\right]}
\]
is a $\CGW(2n,w+1;k)$ if $k$ is even, or a $\CGW(2n,w+1;2k)$ if $k$ is odd.
\end{proposition}

The following generalizes Proposition \ref{prop:add1}.

\begin{proposition}\label{prop:add1gen}
Let $A \in \CGW(n,w_{1};k_{1})$ and $B \in \CGW(n,w_{2};k_{2})$ be such that $AB = BA$. Then the matrix
\[
{\footnotesize\renewcommand{\arraycolsep}{.1cm}
\left[\begin{array}{cc}
A & B \\
-B^{\ast} & A^{\ast}\end{array}\right]}
\]
is a $\CGW(2n,w;k)$ where $w = w_{1}+w_{2}$ and $k = \mathrm{lcm}(k_{1},k_{2},2)$.
\end{proposition}

The conditions for Proposition \ref{prop:add1gen} are met, for example, when $A$ and $B$ are $\zeta$-circulant for some $\zeta \in \mathcal{U}_{k_{1}} \cap \mathcal{U}_{k_{2}}$, although this would just be a very special case of the construction of Theorem \ref{thm:WPPGP-construction}.

\subsubsection{Tensor product type constructions} One of the simplest recursive constructions is via the Kronecker product. For any two matrices $A$ and $B$ having entries with defined multiplication, the Kronecker product of $A$ and $B$ is defined to be the block matrix
\[
A \otimes B = [a_{ij}B].
\]
It is a simple exercise to verify the following.
\begin{proposition}
Let $A \in \CGW(n_{1},w_{1};k_{1})$ and $B \in \CGW(n_{2},w_{2};k_{2})$. Then $A \otimes B \in \CGW(n,w;k)$ where $n = n_{1}n_{2}$, $w = w_{1}w_{2}$ and $k = \mathrm{lcm}(k_{1},k_{2})$.
\end{proposition}

This tensor product construction is generalised by a construction of Di\c{t}\u{a} \cite{Dita}, originally proposed for complex Hadamard matrices. For this construction we require a matrix $A \in \CGW(n,w_{a};k_{a})$ and a set of matrices $\{B_{1},\ldots,B_{n}\}$ with each $B_{i} \in \CGW(m,w_{b,i};k_{b,i})$. Note that each $B_{i}$ must be of the same order.

\begin{proposition}[cf. Proposition 2 \cite{Dita}]\label{prop:dita}
Let $A,B_{1},\ldots,B_{n}$ be as described above. Then
\[
D = {\footnotesize\renewcommand{\arraycolsep}{.1cm}
\left[\begin{array}{cccccc}
a_{11}B_{1} & a_{12}B_{2} & \cdots & a_{1n}B_{n} \\
a_{21}B_{1} & a_{22}B_{2} & \cdots & a_{2n}B_{n} \\
\vdots & \vdots & \ddots & \vdots \\
a_{n1}B_{1} & a_{n2}B_{2} & \cdots & a_{nn}B_{n} \end{array}\right]}
\]
is a $\CGW(mn,w;k)$ where $w = w_{a}\left(\textstyle{\sum}_{i=1}^{n} w_{b,i}\right)$ and $k = \mathrm{lcm}(k_{a},k_{b,1},\ldots,k_{b,n})$.
\end{proposition}

\begin{remark}
If the matrices $B_{1},\ldots,B_{n}$ are all equal, then the matrix $D$ of Proposition \ref{prop:dita} is the Kronecker product $A \otimes B_{1}$.
\end{remark}

A construction of complex Hadamard matrices was introduced by McNulty and Weigert that is a little more general, see \cite[Theorem 3]{McNulty}. It is difficult to see how this construction might be simply generalized for building weighing matrices with parameters not catered for by Proposition \ref{prop:dita}, but in theory it could be so we include it for completeness. The key components are two sets of $q \times q$ unitary matrices $\{L_{1},\ldots,L_{p}\}$ and $\{K_{1},\ldots,K_{p}\}$ such that $K_{i}^{\ast}L_{j}$ is complex Hadamard for all $1 \leq i,j \leq p$, and another complex Hadamard matrix $M = [m_{ij}]$ of order $p$. The result is a $pq \times pq$ complex Hadamard matrix. The matrix constructed takes the form
\[
H = {\footnotesize\renewcommand{\arraycolsep}{.1cm}
\left[\begin{array}{cccccc}
m_{11}K_{1}^{\ast}L_{1} & m_{12}K_{1}^{\ast}L_{2} & \cdots & m_{1p}K_{1}^{\ast}L_{p} \\
m_{21}K_{2}^{\ast}L_{1} & m_{22}K_{2}^{\ast}L_{2} & \cdots & m_{2p}K_{2}^{\ast}L_{p} \\
\vdots & \vdots & \ddots & \vdots \\
m_{p1}K_{p}^{\ast}L_{1} & m_{p2}K_{p}^{\ast}L_{2} & \cdots & m_{pp}K_{p}^{\ast}L_{p} \end{array}\right].}
\]
Restricting entries to $k^{\rm th}$ roots of unity is not a complete barrier to constructing Butson Hadamard matrices and several matrices not of Di\c{t}\u{a} type are constructed in \cite{McNulty}. One of the most useful aspects of this construction is the freedom to use sets of mutually unbiased bases of $\mathbb{C}^{q}$ for the unitary matrices $\{L_{1},\ldots,L_{p}\}$ and $\{K_{1},\ldots,K_{p}\}$.

The obvious drawback to these tensor product type constructions is that the order is typically the product of the orders of the factors, and as a consequence there are significant restrictions on order that can be achieved with this approach. This issue is particularly apparent when constructing real Hadamard matrices. A product of real Hadamard matrices of order $4n$ and $4m$ is necessarily of order divisible by $16$, so to produce a Hadamard matrix of order $4k$ for any odd number $k$, a construction like this fails. Constructions that mitigate this issue are rare, but a method known as weaving, introduced by Craigen \cite{CraigenPhD}, does just this. In general this construction involves orthogonal designs, but it was employed specifically for constructing weighing matrices in \cite{CraigenWeaving95} and it settled some previously undecided existence questions at reasonably small orders and weights.

\subsubsection{Weaving}

The ideas in this section are drawn from the thesis of Craigen \cite{CraigenPhD} but some have since appeared in other published works. The idea of weaving is to knit together weighing matrices of different orders to form a larger one, without relying on a tensor product type construction that forces the order to be the product of the orders of its constituents.  The version of this next Theorem that appears in \cite{CraigenWeaving95} refers only to real weighing matrices, but we give a more general version that applies to CGWs, the proof of which is essentially identical, and is constructive.

\begin{theorem}[cf.{~\cite[Theorem 1]{CraigenWeaving95}}]\label{thm:Weaving}
Let $M = (m_{ij})$ be a $m \times n$ $(0,1)$-matrix with row sums $r_{1},\ldots, r_{m}$ and column sums $c_{1},\ldots,c_{n}$. If for fixed integers $a$ and $b$ there are matrices $A_{i} \in \CGW(r_{i},a;k_{1})$ and $B_{j} \in CGW(c_{j},b;k_{2})$ for $1 \leq i \leq m$ and $1 \leq j \leq n$, then there is a $\CGW(\sigma(M),ab;k)$ where
\[
\sigma(M) = \sum_{i=1}^{m} r_{i} = \sum_{j=1}^{n} c_{j},
\]
and $k = \mathrm{lcm}(k_{1},k_{2})$.

\end{theorem}

\begin{proof}
Construct $W = (W_{ij})$ as an $m \times n$ array of blocks as follows. If $m_{ij} = 0$ set $W_{ij} = 0_{r_{i} \times c_{j}}$. If $m_{ij} = 1$, then $m_{ij}$ is the $p^{\rm th}$ non-zero entry in the $i^{\rm th}$ row, and the $q^{\rm th}$ non-zero entry in the $j^{\rm th}$ column of $M$ for some $p = p(i,j)$ and $q = q(i,j)$. Denote the $p^{\rm th}$ column of $A_{i}$ and the $q^{\rm th}$ row of $B_{j}$ by $A_{i}[\cdot,p]$ and $B_{j}[q,\cdot]$ respectively, and set $W_{ij} = A_{i}[\cdot,p]B_{j}[q,\cdot]$, the rank one $r_{i} \times c_{j}$ matrix. Then $W$ is a square matrix of order $\sigma(M)$, and the entries are in $\mathcal{U}_{k}$. It remains to verify that $WW^{*} = abI_{\sigma(M)}$.

Since $W$ is an $m \times n$ array of blocks, the matrix $WW^{*}$ is expressed as an $m \times m$ array of blocks with the $(i,j)$ block given by
\begin{align*}
\sum_{\ell=1}^{n}W_{i\ell}W_{j\ell}^{*} &= \sum_{\{\ell \, : \, m_{i\ell}=m_{j\ell}=1\}} A_{i}[-,p(i,\ell)]B_{i}[q(i,\ell),-](B_{j}[q(j,\ell),-])^{*}(A_{j}[-,p(j,\ell)])^{*} \\
&= \sum_{\{\ell \, : \, m_{i\ell}=1\}}  \delta_{ij}bA_{i}[-,p(i,\ell)](A_{j}[-,p(j,\ell)])^{*} \\
&= \delta_{ij}b \sum_{p=1}^{r_{i}} A_{i}[-,p](A_{j}[-,p])^{*} \\
&= \delta_{ij}abI_{r_{i}},
\end{align*}

where $\delta_{ij} = 1$ if $i=j$ and $\delta_{ij} = 0$ otherwise. It follows that $W$ is a weighing matrix.
\end{proof}

The conditions of Theorem \ref{thm:Weaving} are such that the weight of the constructed matrix is the product of the two distinct weights of the components, however the order is no longer tied to this condition. The benefit of this is immediately demonstrated in \cite{CraigenWeaving95} by the construction of a $W(66,36)$ using four real weighing matrices - one from each of $W(13,9)$, $W(10,9)$, $W(6,4)$ and $W(4,4)$ - and a $6 \times 13$ matrix $M$ with the required row and column sums. This settled the then open question of existence of a $W(66,36)$.

\begin{example}
{\upshape We can use this technique to build a $\CGW(15,9;3)$, which cannot be constructed through a tensor product. Let
\[
M = {\footnotesize\renewcommand{\arraycolsep}{.1cm}
\left[\begin{array}{ccccc}
1 & 1 & 1 & 0 & 0 \\
0 & 1 & 1 & 1 & 0 \\
0 & 0 & 1 & 1 & 1 \\
1 & 0 & 0 & 1 & 1 \\
1 & 1 & 0 & 0 & 1 \end{array}\right]}
\]
and let
\[
A_{i} = B_{j} = {\footnotesize\renewcommand{\arraycolsep}{.1cm}
\left[\begin{array}{ccc}
1 & 1 & 1  \\
1 & \omega & \omega^{2} \\
1 & \omega^{2} & \omega  \end{array}\right]}
\]
for all $1 \leq i,j \leq 5$, where $\omega = \zeta_{3}$. Then via the method outlined in the proof of Theorem \ref{thm:Weaving} we obtain the matrix
\[
W = {\footnotesize\renewcommand{\arraycolsep}{.1cm}
\left[\begin{array}{ccc|ccc|ccc|ccc|ccc}
1 &  1 &  1 &  1 &  1 &  1 &  1 &  1 &  1 &  0 &  0 &  0 &  0 &  0 &  0 \\
1 &  1 &  1 &  \omega & \omega & \omega & \omega^2 & \omega^2 & \omega^2 & 0 &  0 &  0 &  0 &  0 &  0 \\
1 &  1 &  1 &  \omega^2 & \omega^2 & \omega^2 & \omega & \omega & \omega & 0 &  0 &  0 &  0 &  0 &  0 \\ \hline
0 &  0 &  0 &  1 &  \omega & \omega^2 & 1 &  \omega & \omega^2 & 1 &  1 &  1 &  0 &  0 &  0 \\
0 &  0 &  0 &  1 &  \omega & \omega^2 & \omega & \omega^2 & 1 &  \omega^2 & \omega^2 & \omega^2 & 0 &  0 &  0 \\
0 &  0 &  0 &  1 &  \omega & \omega^2 & \omega^2 & 1 &  \omega & \omega & \omega & \omega & 0 &  0 &  0 \\ \hline
0 &  0 &  0 &  0 &  0 &  0 &  1 &  \omega^2 & \omega & 1 &  \omega & \omega^2 & 1 &  1 &  1 \\
0 &  0 &  0 &  0 &  0 &  0 &  1 &  \omega^2 & \omega & \omega & \omega^2 & 1 &  \omega^2 & \omega^2 & \omega^2 \\
0 &  0 &  0 &  0 &  0 &  0 &  1 &  \omega^2 & \omega & \omega^2 & 1 &  \omega & \omega & \omega & \omega \\ \hline
1 &  \omega & \omega^2 & 0 &  0 &  0 &  0 &  0 &  0 &  1 &  \omega^2 & \omega & 1 &  \omega & \omega^2 \\
1 &  \omega & \omega^2 & 0 &  0 &  0 &  0 &  0 &  0 &  \omega & 1 &  \omega^2 & \omega^2 & 1 &  \omega \\
1 &  \omega & \omega^2 & 0 &  0 &  0 &  0 &  0 &  0 &  \omega^2 & \omega & 1 &  \omega & \omega^2 & 1 \\ \hline
1 &  \omega^2 & \omega & 1 &  \omega^2 & \omega & 0 &  0 &  0 &  0 &  0 &  0 &  1 &  \omega^2 & \omega \\
1 &  \omega^2 & \omega & \omega & 1 &  \omega^2 & 0 &  0 &  0 &  0 &  0 &  0 &  \omega^2 & \omega & 1 \\
1 &  \omega^2 & \omega & \omega^2 & \omega & 1 &  0 &  0 &  0 &  0 &  0 &  0 &  \omega & 1 &  \omega^2 \end{array}\right]}
\]
which is a $\CGW(15,9;3)$.}
\end{example}

Also building on the work in \cite{CraigenPhD}, Craigen and de Launey developed on an idea similar to weaving with the intention of constructing circulant and other group developed CGWs \cite{CraigendeLauney}. Being group developed is an added condition that we don't wish to apply in his paper so we refer the interested reader to the article for more details. However a method of weaving together different objects to form CGWs without necessarily having this property is also described, and a special case of it is used to construct the group developed matrices. Fundamental to this construction is the general concept of an orthogonal set. An \emph{orthogonal set} of weight $w$ is a set of $v \times v$ matrices $\{A_{1},\ldots,A_{n}\}$ such that $A_{i}A_{j}^{*} = 0$ for all $i \neq j$, and there exist positive integers $\lambda_{1},\ldots,\lambda_{n}$ such that
\[
\sum_{i=1}^{n}\lambda_{i}A_{i}A_{i}^{*} = wI_{n}.
\]
The matrices in an orthogonal set can be woven together by summing together the Kronecker products $P_{s}\otimes A_{s}$ provided that $\{P_{1},\ldots,P_{n}\}$ is a set of disjoint $N \times N$ monomial matrices, where $n \leq N$. The result is a $\CGW(nN,w;k)$, if the constituent parts all have entries in $\mathcal{U}_{k}$. Note that the matrices in the orthogonal set are not necessarily CGW matrices, rather the weight of the set is $w$ if there are exactly $w$ non-zero entries in the concatenation of the $r^{\rm th}$ row or column of the matrices $A_{1},\ldots,A_{n}$, for each $1 \leq r \leq n$. This gives a lot of freedom to the construction.

\subsection{Tables of existence}\label{section:Existence}

Harada and Munemasa classified weighing matrices of order up to $15$ in \cite{HarMun12}, building on earlier work of Chan, Rogers and Seberry in \cite{ChaRodSeb}. As such, the question of existence or non-existence of real weighing matrices of order up $15$ is known in all cases. Using a combination of the non-existence results of Section \ref{section:Props}, the constructions of Section \ref{section:Construct}, we attempt to complete tables showing either existence or non-existence of matrices in $\CGW(n,w;k)$ for all $n \leq 15$, $w \leq n$, and $k \in \{2,3,4,5,6\}$. These Tables are presented in Appendix \ref{App:Tables}, with an entry E indicating existence, and N indicating non-existence. Some entries in these tables remain unresolved, they are indicated by a question mark.

In Table \ref{tab:Existence2}, the $k=2$ case is reported, which just compiles results from \cite{HarMun12}.
Tables \ref{tab:Existence3}, \ref{tab:Existence4}, \ref{tab:Existence5}, \ref{tab:Existence6} report on the $k=3,\,4,\,5,\,6$ cases respectively. The as yet undetermined entries which are marked with a ? are all parameters that meet the known existence criteria. In each case, should a CGW exist, we can usually say something about the support matrix. If a $\CGW(12,9;3)$ exists, then up to permutation equivalence its support matrix takes the form $(J_{3} - I_{3}) \otimes J_{3}$. If a $\CGW(13,9;3)$ exists, then its support matrix must be a $\mathrm{SBIBD}(13,9,6)$, which is known to exist. If a $\CGW(15,7;3)$ exists, then its support matrix must be a $\mathrm{SBIBD}(15,7,3)$ which also exists (there is a Hadamard design with these parameters). In these cases, we need to solve the lifting problem.

The restrictions for small $n$ in the $k = 5$ case are such that very little extra analysis is required and almost all parameters are ruled out. In Table \ref{tab:Existence4}, there are only occasionally parameters for which a $\CGW(n,w;4)$ exists and a $\CGW(n,w;2)$ does not. The first we encounter is a $\CGW(10,6;4)$, which can be built from a $\WPPGP(\mathcal{U}_{4},5,1,6)$ where $a = (1,\zeta_{4},1,0,0)$ and $b = (1,-1,-1,0,0)$.

\section{Application: Quantum error-correcting codes}\label{section:Quant}

A classical linear $[n,k,d]_{q}$-\emph{code} $C$ of \emph{minimum distance} $d$ is a $k$-dimensional subspace of $\F_{q}^{n}$, the elements of which are called \emph{codewords}, such that the minimum Hamming distance between any two distinct codewords is $d$. The \emph{rate} of $C$ is the ratio $\frac{k}{n}$.
For fixed parameters $n$ and $k$ a code where $d$ attains the theoretical upper bound is called \emph{optimal}, and one where $d$ does not attain the sharpest known bound, but attains the highest value of any known code, is called \emph{best known}. We refer the reader to \cite{HuffmanPless} for a complete background in coding theory and its applications, and we refer to the expertly maintained webpage at \cite{GrasslCodes} for up to date links to research and tables displaying the best known linear codes for several parameters.

Let $C$ be a $[n,k]_{q^{2}}$ code. The \emph{Hermitian inner product} of codewords $x,y \in C$ is defined by
\[
\langle x,y \rangle = \sum_{i=0}^{n-1}x_{i}y_{i}^{q}.
\]
The \emph{Hermitian Dual} of $C$ is the code
\[
C^{H} = \{x \in C \; \mid \; \langle x,y \rangle = 0 \, \forall \, y \in C\}.
\]
The code $C$ is \emph{Hermitian self-orthogonal} if $C \subseteq C^{H}$, and \emph{Hermitian self-dual} if $C= C^{H}$.

Quantum codes are to quantum information theory what classical codes are to information theory. However, the problem is inherently more difficult due to the postulates of quantum mechanics. We cannot duplicate information by the No-Cloning Theorem \cite{NoClone}, and the observation of a qubit forces it to collapse to a binary state. Shor’s solution \cite{Shor}, is to spread the information of one qubit across the entangled state of several qubits. The following definition is taken from \cite{CalderbankShor96}: A \emph{quantum error-correcting code} is defined to be a unitary mapping (encoding) of $k$ qubits into a subspace of the quantum state space of $n$ qubits such that if any $t$ of the qubits undergo arbitrary decoherence, not necessarily independently, the resulting $n$ qubits can be used to faithfully reconstruct the original quantum state of the $k$ encoded qubits.

Unlike classical codes, quantum codes are usually linear \cite{NRSS06}. For a quantum code with parameters $n$, $k$ and $d$, we typically denote it as an $[[n,k,d]]_{q}$-code. Shor’s model in \cite{Shor} requires the information of one qubit to be spread across nine qubits, so the rate is $1/9$, and it protects against just one error. In \cite{CalderbankShor96} Calderbank and Shor use binary linear codes to build improved quantum codes, and later produce quantum codes capable of correcting multiple errors using group theoretic ideas in \cite{CRSS97}. In \cite{CRSS} it is shown how, given a Hermitian self-orthogonal $[n,k]_{4}$-linear code $C$ such that no codeword in $C^{\perp}\setminus C$ has weight less than $d$, one can construct a quantum $[[n,n-2k,d]]_{2}$-code. Rains \cite{Rainsq2} later established that there are similar applications to Hermitian self-orthogonal $[n,k]_{q^2}$ codes. The following is a restatement of \cite[Corollary 19]{KKKSquant}. See also \cite{AshKnill}.

\begin{theorem}\label{thm:GenQ}
If there exists a linear Hermitian self-orthogonal $[n,k]_{q^2}$ code $C$ such that the minimum weight of $C^{H}$ is $d$, then there exists an $[[n,n-2k,\geq d]]_{q}$ quantum code.
\end{theorem}

\begin{remark}
A quantum code can be $0$-dimensional, and so it is possible to construct a quantum $[[n,0,d]]_{q}$-code given a Hermitian self-dual $[n,n/2,d]_{q^{2}}$ code. See \cite{LisSingh} for details.
\end{remark}

Applications of these results have led to many of the best known constructions of quantum error-correcting codes, and so it is pertinent to study the construction of Hermitian self-orthogonal codes over $\F_{q^2}$. With some restrictions, CGWs provide the perfect tool.

To begin, we observe that when $k = q+1$, we can translate the set of $k^{\rm th}$ roots of unity into $\F_{q^2}$, because $k$ divides $q^{2} - 1$.

The following Propositions formalize and generalize some observations noted in \cite{FFApaper}.

\begin{proposition}\label{prop:IPs}
Let $q$ be a prime power, let $k = q+1$ and let $\alpha$ be a primitive $k^{\rm th}$ root of unity in $\F_{q^{2}}$.
Define the homomorphism $f : \mathcal{U}_{k} \rightarrow \F_{q^{2}}$ so that $f(0) = 0$ and $f(\zeta_{k}^{j}) = \alpha^{j}$ for $j = 0,1,\ldots,q$. Let $x$ be a $\mathcal{U}_{k}$-vector of length $n$ and let $f(x) = [f(x_{i})]_{0 \leq i \leq n-1}$. Then for any $\mathcal{U}_{k}$-vectors $x$ and $y$,
\[
\langle x,y \rangle = 0 \quad \Longrightarrow \quad \langle f(x),f(y)\rangle_{H} = 0.
\]
\end{proposition}

\begin{proof}
By construction, $f(\zeta_{k}^{j})$ is a $k^{\rm th}$ root of unity in the field $\F_{q^{2}}$ for all $0 \leq j \leq q$. Observe that
\[
f(\omega)^{q} = \alpha^{q} = \alpha^{-1} = f(\omega^{\ast}),
\]
for all $\omega \in \mathcal{U}_{k}$.
Then for any $\mathcal{U}_{k}$-vectors $x$ and $y$,
\begin{align*}
\langle f(x),f(y)\rangle_{H} &= \sum_{i=0}^{n-1}f(x_{i})f(y_{i})^{q} \\
&= \sum_{i=0}^{n-1}f(x_{i})f(y_{i}^{\ast}) \\
&= \sum_{i=0}^{n-1}f(x_{i}y_{i}^{\ast})\\
&= f^{+}\left(\sum_{i=0}^{n-1}x_{i}y_{i}^{\ast}\right)\\
&= f^{+}(\langle x, y \rangle).
\end{align*}

Thus if $\langle x,y \rangle = \sum_{j=0}^{k-1} c_{j}\zeta_{k}^{j}$, then $\langle f(x),f(y)\rangle_{H} = \sum_{j=0}^{k-1} c_{j}\alpha^{j(q-1)}$, and so
\[
\langle x,y \rangle = 0 \quad \Longrightarrow \quad \langle f(x),f(y)\rangle_{H} = 0.
\]
\end{proof}

The following Proposition is also now immediate.

\begin{proposition}\label{prop:GenHermCode}
Let $W$ be a $\CGW(n,w;q+1)$ for some prime power $q$ and let $f$ be the homomorphism defined in Proposition \ref{prop:IPs}, with $f(W) = [f(W_{ij})]_{1\leq i,j, \leq n}$. If $w$ is divisible by the characteristic of $\F_{q^{2}}$, then $f(W)$ generates a Hermitian self-orthogonal $F_{q^{2}}$-code.
\end{proposition}

\begin{proof}
Let $x$ and $y$ be distinct rows of $W$. Then $\langle x ,y \rangle = 0$ and so $\langle f(x),f(y)\rangle_{H} = 0$ by Proposition \ref{prop:IPs}. Further, because $x$ is a row of $W$ and so by design, each entry $\alpha$ in $x$ has the property that $\alpha^{\ast}  = \alpha^{-1}$, then $\langle f(x),f(x)\rangle_{H} = 0$ because $w$ is divisible by the characteristic of $\F_{q^{2}}$.
\end{proof}

\begin{remark}
The necessity that a row of $W$ is of weight divisible by the characteristic of $\F_{q^2}$ does not extend to codewords in general. By construction, the entries $\alpha$ in a row of $f(W)$ are all such that $\alpha^{\ast} = \alpha^{q}$, and so $\langle f(\alpha),f(\alpha) \rangle = w$. Other codewords obtained by a linear combination of the rows of $f(W)$ can contain field elements as entries that may not have this property. However, the fact that a linear combination of the rows of $f(W)$ is orthogonal to itself is guaranteed by the properties of an inner product, and the self-orthogonality of the rows of $f(W)$ which form a basis.
\end{remark}

As a consequence of Proposition \ref{prop:GenHermCode} we can use a $\CGW(n,w;k)$ with appropriate weight to build quantum codes for any $k = q + 1$ where $q$ is a prime power, which includes any $k \in \{3,4,5,6,8,9,10\}$.

\begin{remark}
This implication of Proposition \ref{prop:IPs} is one directional, and the converse does not hold. Nevertheless, this relationship is crucial to the classification of matrices in $\BH(18,3)$ via Hermitian self-dual codes over $\F_{4}$ in \cite{HLMT}.
\end{remark}

The propositions above can now be implemented to construct quantum codes.

\begin{example}
{\upshape
Let $W$ be the $\CGW(5,4;3)$ obtained using Berman's construction in Example \ref{ex:Berman}. The code $C$ generated by $W$ is a $[5,2,4]_{4}$ code, and the hermitian dual $C^{H}$ is a $[5,3,3]_{4}$ code. Applying Theorem \ref{thm:GenQ}, we construct a $[[5,1,3]]_{2}$ quantum error-correcting code, which is optimal.
}
\end{example}

Since the construction of quantum codes is generalized by Theorem \ref{thm:GenQ}, our intention now is to generalize the propositions above.

\begin{example}
{\upshape
Let $W$ be the $\CGW(10,9;4)$ obtained using the Seberry and Whiteman construction in Example \ref{ex:SW}. The code $C$ generated by $W$ is a $[10,5,4]_{9}$ code, and is Hermitian self-dual. We apply Theorem \ref{thm:GenQ} and construct a $[[10,0,4]]_{3}$ quantum error-correcting code.
}
\end{example}

Any $\BH(n,4)$ where $3 \mid n$ may be used to construct ternary quantum codes in this manner. This is particularly useful because $\BH(n,4)$ matrices are plentiful. For example, there are exactly 319 equivalence classes in $\BH(12,4)$, see \cite[Theorem 6.1]{LSO}.

\begin{example}
{\upshape
 For example, let
\[
H = {\footnotesize\renewcommand{\arraycolsep}{.1cm}
\left[\begin{array}{cccccc}
1 & i & 1 & 1 & 1 & - \\
1 & 1 & i & - & 1 & 1 \\
i & 1 & 1 & 1 & - & 1 \\
1 & - & 1 & - & - & i \\
1 & 1 & - & i & - & - \\
- & 1 & 1 & - & i & - \end{array}\right]}.
\]
The Code $C$ generated by $H$ is a Hermitian self-dual $[6,3,4]_{9}$ code, which constructs a $[[6,0,4]]_{3}$ quantum error-correcting code.
}
\end{example}

\begin{example}\label{ex:Best5}
{\upshape
By Proposition \ref{prop:GenHermCode}, we can use a $\CGW(n,w;6)$ with $w$ divisible by $5$ to construct a Hermitian self-orthogonal code over $\F_{25}$. As an example, we take the $\BH(25,6)$ constructed via \cite[Theorem 1.4.41]{FerencThesis}, and construct a $[25,9,13]_{25}$ Hermitian self-orthogonal code $C$. The Hermitian dual $C^{H}$ is a $[25,16,6]_{25}$ code, and so we construct a $[[25,7,6]]_{5}$ quantum code. This has larger minimum weight than the current best known $[[25,7]]_{2}$ quantum code, which has minimum weight $5$, according to \cite{GrasslCodes}.
}
\end{example}

\section{Computational results}\label{section:results}

In Table \ref{tab:QTab} we list the parameter of the quantum codes constructed that are, according to the information available to us, at least as good or better than the best known quantum codes. It is difficult to compare the codes constructed to others that may be known, as there does not appear to be any database comparable to \cite{GrasslCodes} that caters for quantum $q$-ary codes in general. Recently, the authors of \cite{QaryQCodes} have introduced a database, but at least for now it is not completely populated for all parameters. At the time of writing, the only $[[n,k]]_{q}$ code in this database that is comparable to a code in Table \ref{tab:QTab} is a $[[24,0,6]]_{3}$ code; we found a $[[24,0,9]]_{3}$ code. For this reason, the parameters of codes constructed here are often compared to the best known $[[n,k]]_{2}$ codes listed in \cite{GrasslCodes}.

\begin{center}
\begin{table}[h]
\centering
\begin{tabular}{c|c|c|c}
Source matrix & Self orthogonal $[n,k,d]_{q^2}$ code & New $[[n,k,d]]_{q}$ code & Best known $[[n,k]]_{2}$ from \cite{GrasslCodes} \\\hline
$\CGW(5,4;3)$ & $[5,2,4]_{4}$ &    $[[5,1,3]]_{2}$ & $[[5,1,3]]_{2}$ \\
$\BH(6,4)$ & $[6,3,4]_{9}$ &    $[[6,0,4]]_{3}$ & $[[6,0,4]]_{2}$ \\
$\BH(9,10)$ & $[9,4,6]_{81}$ &    $[[9,1,5]]_{9}^{\ast}$ & $[[9,1,3]]_{2}$ \\
$\CGW(10,9;4)$ & $[10,5,4]_{9}$ &    $[[10,0,4]]_{3}$ & $[[10,0,4]]_{2}$ \\
$\BH(10,6)$ & $[10,5,5]_{25}$ &     $[[10,0,5]]_{5}^{\ast}$ & $[[10,0,4]]_{2}$ \\
$\BH(10,5)$ & $[10,5,6]_{16}$ &     $[[10,0,6]]_{4}^{\ast}$ & $[[10,0,4]]_{2}$ \\
$\CGW(12,10;6)$ & $[12,6,6]_{25}$ &     $[[12,0,6]]_{5}$ & $[[12,0,6]]_{2}$ \\
$\BH(14,8)$ & $[14,7,8]_{49}$ &     $[[14,0,8]]_{7}^{\ast}$ & $[[14,0,6]]_{2}$ \\
$\BH(18,4)$ & $[18,9,8]_{9}$ &     $[[18,0,8]]_{3}$ & $[[18,0,8]]_{2}$ \\
$\BH(20,6)$ & $[20,10,8]_{25}$ &     $[[20,0,8]]_{5}$ & $[[20,0,8]]_{2}$ \\
$\BH(20,5)$ & $[20,9,8]_{16}$ &      $[[20,2,6]]_{4}$ & $[[20,2,6]]_{2}$ \\
$\CGW(20,9;4)$ & $[20,8,9]_{9}$ &      $[[20,4,6]]_{3}$ & $[[20,4,6]]_{2}$ \\
$\CGW(21,16;3)$ & $[21,3,16]_{4}$ &      $[[21,15,3]]_{2}$ & $[[21,15,3]]_{2}$ \\
$\BH(24,4)$ & $[24,12,9]_{9}$ &     $[[24,0,9]]_{3}$ & $[[24,0,8]]_{2}$ \\
$\BH(25,6)$ & $[25,9,13]_{25}$ &      $[[25,7,6]]_{5}$ & $[[25,7,5]]_{2}$ \\
$\CGW(26,25;6)$ & $[26,5,22]_{25}$ &      $[[26,16,6]]_{5}^{\ast}$ & $[[26,16,4]]_{2}$ \\
$\BH(30,4)$ & $[30,15,12]_{9}$ &     $[[30,0,12]]_{3}$ & $[[30,0,12]]_{2}$ \\
$\BH(36,3)$ & $[36,18,12]_{4}$ &     $[[36,0,12]]_{2}$ & $[[36,0,12]]_{2}$ \\
$\BH(42,4)$ & $[42,21,14]_{9}$ &     $[[42,0,14]]_{3}$ & $[[42,0,12]]_{2}$ \\
\end{tabular}
\caption{New quantum codes}\label{tab:QTab}
\end{table}
\end{center}

\begin{remark}
All of the $[[n,k]]_{q}$ codes listed in Table \ref{tab:QTab} have a minimum distance at least as large as any known $[[n,k]]_{2}$ code according to \cite{GrasslCodes}. The codes marked with an asterisk listed in Table \ref{tab:QTab} are examples of $[[n,k]]_{q}$ quantum codes with a minimum distance that surpasses the known upper bound for a corresponding $[[n,k]]_{2}$ code.  The matrices used to build the codes in Table \ref{tab:QTab} come from a variety of sources, many of which are from constructions outlined in this paper. Many of source matrices are Butson matrices, taken from existing databases such as the online database of complex Hadamard matrices at \cite{Karol}.
\end{remark}

\subsection{Concluding remarks}

The computations of this section are not the result of exhaustive searches, as we do not have access to any convenient database of matrices to search through. Nor have we attempted to use any coding theory methods to either extend the codes we found, or to search for good subcodes. The codes with parameters listed in Table \ref{tab:QTab} are the results of ``proof of concept'' experimentation using matrices we could either construct using some of the methods described in this paper, or matrices that could be easily accessed through online sources. The purpose is to demonstrate that good quantum codes can be constructed. A complete computational survey of codes constructible with these tools is beyond the scope of this paper, but the evidence presented here suggests that many good codes may be found with this approach. Mostly Butson matrices were used as source matrices because they can be easier to find in databases. A large database of CGWs with different parameters would be a worthwhile development. Finally, we note that the Tables in Appendix \ref{App:Tables} below are incomplete, and any contributions to their completion are very welcome.

\section*{Declaration of competing interest}
The author declares that they have no known competing financial interests or personal relationships that could have
appeared to influence the work reported in this paper.

\section*{Acknowledgements}
The author thanks Rob Craigen, Wolf Holzmann and Hadi Kharaghani for sharing complex Golay sequences computed in \cite{CHK} which we used to build matrices in $\BH(n,4)$, and subsequently $[[n,k]]_{3}$ quantum codes.

\bibliographystyle{abbrv}
\flushleft{
\bibliography{MyBiblio}
}
\newpage

\appendix

\section{Appendix - Tables of existence of $\CGW(n,w;k)$}\label{App:Tables}

{\footnotesize{

\begin{center}
\begin{table}[h]
\footnotesize\centering
\begin{tabular}{c|c|c|c|c|c|c|c|c|c|c|c|c|c|c|c|c|c|c|c|c}
$n \setminus w$ & 1 & 2 & 3 & 4 & 5 & 6 & 7 & 8 & 9 & 10 & 11 & 12 & 13 & 14 & 15 \\\hline
1 & E & & & & & & & & & & & & & & \\\hline
2 & E & E & & & & & & & & & & & & & \\\hline
3 & E & N & N & & & & & & & & & & & & \\\hline
4 & E & E & E & E & & & & & & & & & & & \\\hline
5 & E & N & N & N & N & & & & & & & & & & \\\hline
6 & E & E & N & E & E & N & & & & & & & & & \\\hline
7 & E & N & N & E & N & N & N & & & & & & & & \\\hline
8 & E & E & E & E & E & E & E & E & & & & & & & \\\hline
9 & E & N & N & N & N & N & N & N & N & & & & & & \\\hline
10 & E & E & N & E & E & N & N & E & E & N & & & & & \\\hline
11 & E & N & N & E & N & N & N & N & N & N & N & & & & \\\hline
12 & E & E & E & E & E & E & E & E & E & E & E & E & & & \\\hline
13 & E & N & N & E & N & N & N & N & E & N & N & N & N & & \\\hline
14 & E & E & N & E & E & N & N & E & E & E & N & N & E & N & \\\hline
15 & E & N & N & E & N & N & N & N & E & N & N & N & N & N & N \\\hline
\end{tabular}
\caption{$k=2$}\label{tab:Existence2}
\end{table}
\end{center}

\begin{center}
\begin{table}[h]
\footnotesize\centering
\begin{tabular}{c|c|c|c|c|c|c|c|c|c|c|c|c|c|c|c|c|c|c|c|c}
$n \setminus w$ & 1 & 2 & 3 & 4 & 5 & 6 & 7 & 8 & 9 & 10 & 11 & 12 & 13 & 14 & 15 \\\hline
1 & E & & & & & & & & & & & & & & \\\hline
2 & E & N & & & & & & & & & & & & & \\\hline
3 & E & N & E & & & & & & & & & & & & \\\hline
4 & E & N & N & N & & & & & & & & & & & \\\hline
5 & E & N & N & E & N & & & & & & & & & & \\\hline
6 & E & N & E & N & N & E & & & & & & & & & \\\hline
7 & E & N & N & N & N & N & N & & & & & & & & \\\hline
8 & E & N & N & N & N & N & E & N & & & & & & & \\\hline
9 & E & N & E & N & N & N & N & N & E & & & & & & \\\hline
10 & E & N & N & E & N & N & N & N & N & N & & & & & \\\hline
11 & E & N & N & N & N & N & N & N & N & N & N & & & & \\\hline
12 & E & N & E & N & N & E & N & N & ? & N & N & E & & & \\\hline
13 & E & N & N & N & N & N & N & N & ? & N & N & N & N & & \\\hline
14 & E & N & N & N & N & N & ? & N & N & N & N & N & E & N & \\\hline
15 & E & N & N & E & N & N & ? & N & E & N & N & E & N & N & N \\\hline
\end{tabular}
\caption{$k=3$}\label{tab:Existence3}
\end{table}
\end{center}

\begin{center}
\begin{table}[h]
\footnotesize\centering
\begin{tabular}{c|c|c|c|c|c|c|c|c|c|c|c|c|c|c|c|c|c|c|c|c}
$n \setminus w$ & 1 & 2 & 3 & 4 & 5 & 6 & 7 & 8 & 9 & 10 & 11 & 12 & 13 & 14 & 15 \\\hline
1 & E & & & & & & & & & & & & & & \\\hline
2 & E & E & & & & & & & & & & & & & \\\hline
3 & E & N & N & & & & & & & & & & & & \\\hline
4 & E & E & E & E & & & & & & & & & & & \\\hline
5 & E & N & N & N & N & & & & & & & & & & \\\hline
6 & E & E & N & E & E & E & & & & & & & & & \\\hline
7 & E & N & N & E & N & N & N & & & & & & & & \\\hline
8 & E & E & E & E & E & E & E & E & & & & & & & \\\hline
9 & E & N & N & N & N & N & N & N & N & & & & & & \\\hline
10 & E & E & N & E & E & E & N & E & E & E & & & & & \\\hline
11 & E & N & N & E & N & N & N & N & N & N & N & & & & \\\hline
12 & E & E & E & E & E & E & E & E & E & E & E & E & & & \\\hline
13 & E & N & N & E & N & N & N & ? & E & N & N & N & N & & \\\hline
14 & E & E & N & E & E & E & ? & E & E & E & N & ? & E & E & \\\hline
15 & E & N & N & E & ? & N & N & ? & E & N & N & N & N & N & N \\\hline
\end{tabular}
\caption{$k=4$}\label{tab:Existence4}
\end{table}
\end{center}

\begin{center}
\begin{table}[h]
\footnotesize\centering
\begin{tabular}{c|c|c|c|c|c|c|c|c|c|c|c|c|c|c|c|c|c|c|c|c}
$n \setminus w$ & 1 & 2 & 3 & 4 & 5 & 6 & 7 & 8 & 9 & 10 & 11 & 12 & 13 & 14 & 15 \\\hline
1 & E & & & & & & & & & & & & & & \\\hline
2 & E & N & & & & & & & & & & & & & \\\hline
3 & E & N & N & & & & & & & & & & & & \\\hline
4 & E & N & N & N & & & & & & & & & & & \\\hline
5 & E & N & N & N & E & & & & & & & & & & \\\hline
6 & E & N & N & N & N & N & & & & & & & & & \\\hline
7 & E & N & N & N & N & N & N & & & & & & & & \\\hline
8 & E & N & N & N & N & N & N & N & & & & & & & \\\hline
9 & E & N & N & N & N & N & N & N & N & & & & & & \\\hline
10 & E & N & N & N & E & N & N & N & N & E & & & & & \\\hline
11 & E & N & N & N & N & N & N & N & N & N & N & & & & \\\hline
12 & E & N & N & N & N & N & N & N & N & N & E & N & & & \\\hline
13 & E & N & N & N & N & N & N & N & N & N & N & N & N & & \\\hline
14 & E & N & N & N & N & N & N & N & N & N & N & N & N & N & \\\hline
15 & E & N & N & N & N & N & N & N & N & N & N & N & N & N & N \\\hline
\end{tabular}
\caption{$k=5$}\label{tab:Existence5}
\end{table}
\end{center}

\begin{center}
\begin{table}[h]
\footnotesize\centering
\begin{tabular}{c|c|c|c|c|c|c|c|c|c|c|c|c|c|c|c|c|c|c|c|c}
$n \setminus w$ & 1 & 2 & 3 & 4 & 5 & 6 & 7 & 8 & 9 & 10 & 11 & 12 & 13 & 14 & 15 \\\hline
1 & E & & & & & & & & & & & & & & \\\hline
2 & E & E & & & & & & & & & & & & & \\\hline
3 & E & N & E & & & & & & & & & & & & \\\hline
4 & E & E & E & E & & & & & & & & & & & \\\hline
5 & E & N & N & E & N & & & & & & & & & & \\\hline
6 & E & E & E & E & E & E & & & & & & & & & \\\hline
7 & E & N & E & E & N & N & E & & & & & & & & \\\hline
8 & E & E & E & E & E & E & E & E & & & & & & & \\\hline
9 & E & N & E & E & N & N & ? & N & E & & & & & & \\\hline
10 & E & E & E & E & E & ? & ? & E & E & E & & & & & \\\hline
11 & E & N & E & E & N & N & ? & N & ? & N & N & & & & \\\hline
12 & E & E & E & E & E & E & E & E & E & E & E & E & & & \\\hline
13 & E & N & E & E & N & N & ? & N & E & N & N & ? & E & & \\\hline
14 & E & E & E & E & E & E & E & E & E & E & ? & ? & E & E & \\\hline
15 & E & N & E & E & N & N & E &  N & E & N & N & E & ? & N & N \\\hline
\end{tabular}
\caption{$k=6$}\label{tab:Existence6}
\end{table}
\end{center}

}
}

\end{document}